\documentclass[reqno]{amsart}
\usepackage{graphicx,cite}
\usepackage{amsfonts,amsmath,amsthm,amssymb,latexsym}%
\usepackage[colorlinks,plainpages,citecolor=magenta, linkcolor=blue, bookmarksnumbered, backref, draft=false]{hyperref}

\newcommand{\Real}{\mathbb R}

\renewcommand{\Re}{\mathop{\rm Re}\nolimits}

\newcommand{\tr}{\mathop{\rm tr}\nolimits}

\newcommand{\cF}{\mathcal{F}}

\newtheorem{lem}{Lemma}

\newtheorem{thm}{Theorem}
\newtheorem{rem}{Remark}
\newtheorem{cor}{Corollary}
\newtheorem{pro}{Proposition}
\newtheorem{prp}{Property}

\numberwithin{equation}{section}
\begin{document}
\title[Scattering in Quantum Graphs]{Scattering in Quantum Graphs with Scale-Invariant Vertex Couplings: \\ Resonances, Gaps and (Quasi-)Periodic Transmission}

\author{Khrystyna Buhrii}%
\address[KB]{Ivan Franko National University of Lviv,
	1 Universytetska st., 79602 Lviv, Ukraine}
\email{khrystyna.buhrii@lnu.edu.ua}%

\author{Yuriy Golovaty} \address[YG]{Ivan Franko National University of Lviv,
	1 Universytetska st., 79602 Lviv, Ukraine \and Ukrainian Catholic University, 2a Kozelnytska str., 79026, Lviv, Ukraine}%
\email{yuriy.golovaty@lnu.edu.ua, yuriy.golovaty@ucu.edu.ua}

\author{Rostyslav Hryniv}%
\address[RH]{%Faculty of Applied Sciences,
	Ukrainian Catholic University, 2a Kozelnytska str., 79026, Lviv, Ukraine \and
	%Department of Mathematics and Natural Sciences, the
	University of Rzesz\'{o}w, 1 Pigonia str., 35-310 Rzesz\'{o}w, Poland}%
\email{rhryniv@ucu.edu.ua, rhryniv@ur.edu.pl}

\begin{abstract}
    We study scattering on quantum graphs that consist of a channel with periodically attached resonators under scale-invariant vertex couplings. For this model, we derive explicit formulas for the transmission probability and analyse how it depends on the geometric and coupling parameters.

	Contrary to standard one-dimensional scattering, where the potential barrier becomes transparent at high energy, here the transmission probability does not approach unity; instead, it is periodic or quasi-periodic, with infinitely many energies of complete reflection and of perfect transmission persisting at arbitrarily high energy.

    We further show that the model exhibits strong transmission suppression near the anti-resonant frequencies, resulting in pronounced spectral gaps. The width and structure of these gaps depend on the number of resonators and the coupling parameters. As a result, such quantum graphs can be used to engineer transport properties and to tune spectral filtering.
\end{abstract}
\maketitle

%\tableofcontents

%% =================================
%
\section{Introduction}
%
%% =================================

In standard one-dimensional scattering, a localised potential or point interaction
becomes transparent at high energy: the transmission probability tends to one as the
wavenumber~$k$ grows. We show that this fails for a class of quantum-graph models in
which a line carries resonators attached through scale-invariant vertex couplings,
with perfect transmission and complete reflection occurring at infinitely many
arbitrarily high wavenumbers~$k$. Quantum graphs provide the natural framework for
such systems; they describe, e.g., the evolution of particles in quantum wires,
photonic waveguides, and semiconductor nanostructures
\cite{KuchmentKunyansky1999, Kuchment2004, Exner2007LeakyQuantumGraphs,
BerkolaikoKuchment2013}. Moreover, quantum graphs frequently yield exactly solvable
models that enable rigorous analytical treatment.

In this paper, we study one-dimensional scattering in a channel coupled to a finite
locally periodic array of resonators, modelled as side-attached pendant edges
supporting discrete resonant modes. These edges are connected to the main channel via
scale-invariant vertex coupling conditions \cite[p.~22]{BerkolaikoKuchment2013}. The
self-adjoint vertex conditions that make a quantum graph Hamiltonian well defined were
classified by Kostrykin and Schrader \cite{KostrykinSchrader1999}; they can be realised as norm-resolvent limits of graphs with short internal edges carrying potentials \cite{CheonExnerTurek2010}. The couplings used here form a distinguished scale-invariant subclass depending on several tunable parameters, and by adjusting these parameters one can control the interaction between the propagating wave and the resonators, and
thereby shape the transmission profile. We derive a single closed-form expression for
the transmission probability, valid for any number of resonators, as a function of
the particle energy, providing a flexible framework for designing quantum devices such
as band-pass filters.

This behaviour stands in sharp contrast to scattering in classical media. For locally
periodic potentials on a line---both regular potentials and $\delta$-combs---the
transmission probability tends to one at high energy \cite{GriffithsSteinke}. What
distinguishes the present model is the scale invariance of the vertex couplings:
lacking an intrinsic length scale, they do not become transparent, and the
transmission probability remains periodic or quasi-periodic in the wavenumber; see
Fig.~\ref{Fig1}. A similar persistence of high-energy scattering anomalies was
recently observed for one-dimensional systems with dipole-type singularities
\cite{Gadella2026}.

\begin{figure}[b!]
    \centering
    \includegraphics[scale = 0.3]{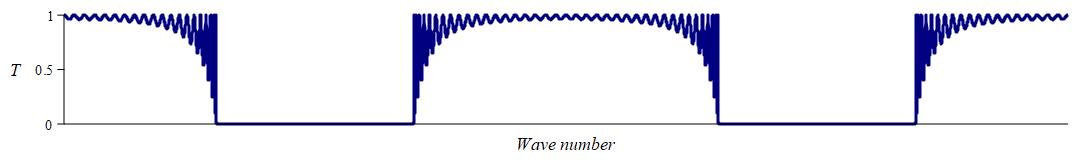}
    \caption{Energy-periodic transmission probability for a channel with $N=20$ resonators: formation of pass and stop bands.}\label{Fig1}
\end{figure}

Our model has a direct predecessor on the line, the $\delta'$-comb. The transmission
probability through a finite locally periodic sequence of $\delta'$ interactions was
shown in \cite{GolovatyHrynivLavrynenko2025} to be periodic in the wavenumber,
exhibiting the same loss of high-energy transparency seen here. These $\delta'$
interactions are not merely formal: they arise as norm-resolvent limits of
Schr\"odinger operators with short-range, dipole-like potentials, as established in
\cite{GolovatyMankoDopovidi, GolovatyMankoUMB, GolovatyHryniv2010,
GolovatyHrynivProcEdinburgh2013, Zolotaryuk08, Zolotaryuk10}. The scale-invariant
vertex couplings on graphs that we use admit an analogous justification: they were
obtained as norm-resolvent limits of Schr\"odinger operators with localised dipole
potentials in \cite{Manko2010, ExnerManko2013, GolovatyAHP2023}. The present paper
replaces the $\delta'$ point interactions by compact resonators attached through
these couplings, yielding a quantum-graph generalisation of the $\delta'$-comb and a
class of exactly solvable scattering models.

Scattering by a single resonator goes back to \cite{ExnerSeresova1994}, where the
vertex conditions ensured continuity of the wave function; the scale-invariant
couplings considered here are more general and need not be continuous. 
The $\delta'$ couplings at multi-link graph vertices, and their
identification with limiting geometric scatterers, were introduced by Exner
\cite{Exner1996}; Floquet--Bloch analysis of the infinite periodic case shows that
the $\delta'_s$ lattice always has infinitely many gaps with bounded widths, while the
gap structure of the $\delta$ lattice depends on number-theoretic properties of the
spacing ratio. More broadly, decorating a graph by attaching compact subgraphs opens
spectral gaps around the eigenvalues of the decoration \cite{SchenkerAizenman2000},
and for periodic chain graphs the resulting band pattern depends on the
commensurability of the edge lengths \cite{BaradaranExnerTater2022}, in agreement
with the role of the ratio $\ell/h$ in our model.

Spectral filtering via scale-invariant couplings has been studied from single
junctions toward the periodic-array setting of the present paper. At a single
Y-junction, $\delta'$-type vertex conditions produce a high-energy blockade between
connected lines \cite{CheonExnerTurek2009}. In a quantum star graph, F\"ul\"op--Tsutsui
couplings yield a tunable band-pass filter whose passband position is controlled by an
external potential on a dedicated line \cite{TurekCheon2013}. For a single junction
coupling a compact graph~$\Gamma$ to an input--output line, transmission resonances
occur at the eigenvalues of~$\Gamma$ and the structure acts as a spectral filter;
conditions under which the passband is exactly flat have been characterised in
\cite{TurekCheon2013JMP, Turek2017}. Our contribution is to pass from a
single decorated junction to a periodic array of $N$ such resonators: we derive
$T_N$ in closed form for any~$N$, establish the persistence of complete reflection
and perfect transmission to arbitrarily large wavenumbers, and analyse the emergence
of the band-and-gap structure as $N\to\infty$. Related filtering behaviour, with narrow near-unity transmission peaks in compact quantum graphs chained in series, has been verified experimentally in microwave networks \cite{AkhshaniBialousSirko2023}. The analysis of such periodic arrays rests on a classical tool of one-dimensional scattering theory.

The transfer-matrix method underlying
our analysis, its composition property, and its application to locally periodic
potentials are treated in \cite{Sanchez2012, Mostafazadeh2020}; explicit formulas for
the transmission through an arbitrary number of identical cells, in terms of the
single-cell scattering data, are classical \cite{Cvetic1981, Erdos1982, GriffithsSteinke}.
For these formulas to apply, one needs the scattering data of a single cell --- here a single resonator --- which for the scale-invariant couplings considered in this paper we compute in closed form.

Our main results concern the transmission probability $T_N$ of an array of $N$ resonators. We first reduce the scattering on each decorated junction to an energy-dependent point interaction on the line, characterised by the resonator's Dirichlet-to-Neumann map; substituting the resulting single-cell data into the locally periodic transfer-matrix formula yields a closed-form expression for $T_N$
valid for every $N$ (Theorem~\ref{thm:TN}). This expression can be recast in a particularly transparent form involving the propagation phase between neighbouring junctions and the resonator phase, on which the subsequent analysis rests. We then derive a number of symmetries
and invariants of $T_N$ with respect to the model parameters, and characterise its
high-energy asymptotics. In particular, we establish the existence of infinitely many
interlaced wavenumbers of complete reflection and perfect transmission, which persist
to arbitrarily high energy (Theorem~\ref{thm:high-energy}). Furthermore, we analyse
the behaviour of $T_N$ near the complete-reflection wavenumbers---where the
transmission suppression sharpens as the number of resonators increases---and examine
its periodic or quasi-periodic nature. Finally, we demonstrate the effect of the
system parameters on the scattering properties and illustrate the emergence of the
spectral band-and-gap structure of the limiting periodic operator in the
infinite-array limit ($N\to\infty$).

The paper is organised as follows. Section~\ref{sec:problem-formulation} describes the
scattering system, a quantum graph consisting of a line with attached resonators.
Section~\ref{sec:single-resonator} treats a single resonator and derives its
$k$-dependent transfer matrix. Section~\ref{sec:local-periodic} analyses an array of
$N$ resonators and derives the closed-form formula for $T_N$.
Section~\ref{sec:properties} studies the properties of $T_N$: its symmetries,
(quasi-)periodicity, high-energy resonances, gap asymptotics, and the high-energy
proximity of Robin and Neumann transmission. Finally,
Section~\ref{sec:asymptotics} relates the large-$N$ behaviour of $T_N$ to the
band-and-gap structure of the limiting periodic operator, and illustrates the effect
of the parameters in the vertex condition on the transmission profile.

%% =================================
%
\section{Problem formulation}\label{sec:problem-formulation}
%
%% =================================

We study quantum scattering in a thin, single-channel waveguide with a finite regular array of side-attached resonators --- finite edges acting as local resonant elements. Such systems are rigorously modelled, in the limit of vanishing waveguide width, by scattering problems on a metric graph~$G$ equipped with a 
Hamiltonian on each edge and suitable vertex conditions~\cite{DellAntonio2012}  ---  i.e., by scattering problems on quantum graphs~\cite{KuchmentZeng2001, 	ExnerPost2005, ExnerPost2009, Post2012}; see~\cite[p.13]{BerkolaikoKuchment2013} for basics of quantum graph theory.

To construct the model, we consider an equally spaced sequence of $N$ points on a line and attach segments of equal length to these points. 
The resulting  graph $G$ is shown in Fig.~\ref{Fig2}.  The vertices of degree $3$ are denoted by $v_1,\dots,v_N$, while the vertices of degree $1$  are denoted by $u_1,\dots,u_N$. The edges $(u_j, v_j)$ represent resonators. The graph $G$ is non-compact due to the presence of two infinite edges, $e_{in}$ and $e_{out}$, called leads, which are attached to the vertices  $v_1$ and $v_N$, respectively. Finally, we introduce two geometric parameters: the length $l$ of each resonator and the distance $h$ between neighboring attachment points.

\begin{figure}[h]
  \centering
  \includegraphics[scale = 0.5]{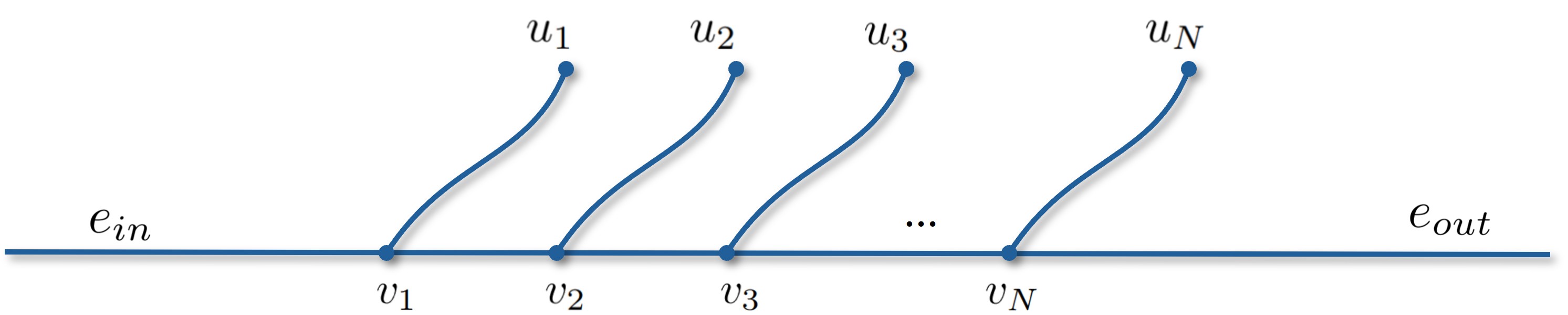}
  \caption{A quantum graph consisting of a scattering channel with resonators}\label{Fig2}
\end{figure}

A function $\psi$ on the graph $G=(V,E)$ is defined as a collection $\{\psi_e\}_{e\in E}$ of functions $\psi_e \colon e \to \mathbb{C}$ on its edges. 
For each vertex $v$ of degree~$3$, we write 
$$
\psi(v)=(\psi_1(v),\psi_2(v),\psi_3(v))
$$ 
for the vector of boundary values (limits of $\psi$ along the three incident edges, indexed as in Fig.~\ref{fig:condAB}) and $$
\psi'(v)=(\psi_1'(v),\psi_2'(v),\psi_3'(v)),
$$
for the vector of 
one-sided derivatives, each $\psi_j'(v)$  taken in the direction from the vertex $v$ into the respective incident edge.

\begin{figure}[t]
	\centering
	\includegraphics[scale = 0.55]{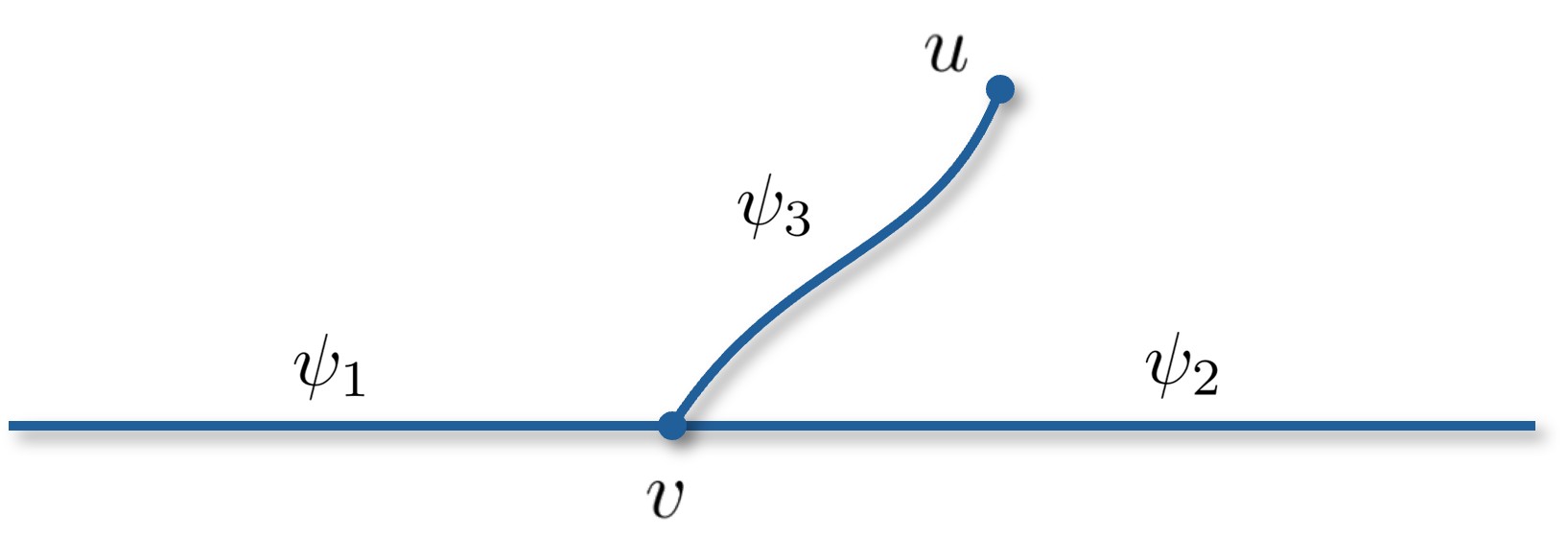}
	\caption{Numbering of the edges adjacent to vertex $v$ of degree~$3$: indices $1$ and $2$ denote the edges belonging to the channel, and index $3$ denotes the resonator edge.}
	\label{fig:condAB}
\end{figure}

On the metric graph $G$, we consider the Hamiltonian $H$ acting as $H\psi = -\psi''$ on each edge. The domain $D(H)$ consists of functions $\psi = \{\psi_e\}_{e\in E}$ in the Sobolev space $\bigoplus_{e\in E} H^2(e)$ that satisfy specific conditions at the vertices to ensure self-adjointness of $H$. Namely, at each vertex $u_j$, we impose the Robin boundary condition
\begin{equation}\label{eq:boundary-conditions-Robin}
	\alpha \psi(u_j) + \beta \psi'(u_j)=0, \quad j=1,\dots,N,
\end{equation}
with $\alpha, \beta \in \mathbb{R}$ and $|\alpha|+|\beta|>0$. At each internal vertex $v$ of degree $3$, we consider one of the following two types of scale-invariant coupling conditions~\cite{BerkolaikoKuchment2013}, parameterized by a fixed vector $\theta = (\theta_1, \theta_2, \theta_3) \in \mathbb{R}^3$:
\begin{itemize}
	\item[\emph{Type I}] \textit{(simple resonance):}
	\begin{equation}\label{eq:coupling-conditions-A}
		\frac{\psi_1(v)}{\theta_1} = \frac{\psi_2(v)}{\theta_2} = \frac{\psi_3(v)}{\theta_3},
		\qquad
		\theta_1 \psi'_1(v) + \theta_2 \psi'_2(v) + \theta_3 \psi'_3(v) = 0.
	\end{equation}
	
	\item[\emph{Type II}] \textit{(double resonance):}
	\begin{equation}\label{eq:coupling-conditions-B}
		\frac{\psi'_1(v)}{\theta_1} = \frac{\psi'_2(v)}{\theta_2} = \frac{\psi'_3(v)}{\theta_3},
		\qquad
		\theta_1 \psi_1(v) + \theta_2 \psi_2(v) + \theta_3 \psi_3(v) = 0.
	\end{equation} 
\end{itemize}
Whenever a component $\theta_j$ vanishes, the corresponding boundary value $\psi_j(v)$ in~\eqref{eq:coupling-conditions-A} (or the derivative $\psi'_j(v)$ in~\eqref{eq:coupling-conditions-B}) is understood to be zero. The terminology ``simple'' and ``double'' resonance refers to the spectral properties of the approximating local operators, as introduced in~\cite{Manko2010}. These coupling conditions arise naturally in the limits of narrow waveguides where dipoles are approximated by $\delta'$-potentials with distinct resonant profiles (see~\cite{ExnerManko2013} for further details and \cite{GolovatyAHP2023} for the general case). 

As follows from the general theory of quantum graphs, the Hamiltonian~$H$ defined above is self-adjoint~\cite{KostrykinSchrader1999, BerkolaikoKuchment2013}.
The scattering process at energy $E = k^2 > 0$ is governed by the \textit{wave function}  ---  a collection $\psi(\cdot,k) = \{\psi_e(\cdot,k)\}_{e\in E}$ of solutions to the equations 
\begin{equation}\label{eq:scattering-problem-definition}
	\psi_e'' + k^2\psi_e=0  \quad \text{on each edge } e\in E,
\end{equation}
which satisfies the boundary conditions~\eqref{eq:boundary-conditions-Robin}  and one of the vertex 
conditions~\eqref{eq:coupling-conditions-A} or~\eqref{eq:coupling-conditions-B} at the internal vertices.
We parametrise the leads by the coordinate $x$ on the real line: the input lead~$e_{\mathrm{in}}$ corresponds to $x \in (-\infty, 0]$ 
(with $v_1$ at $x=0$), and the output lead $e_{\mathrm{out}}$ corresponds to $x \in [L, +\infty)$, where $L=(N-1)h$ is the 
position of $v_N$. 

Among all wave functions, the physically relevant one is the \textit{scattering solution}  ---  the  wave function whose asymptotic behavior on the leads takes the form 
	\begin{alignat}{2}\label{AsympIn}
	&\psi(x,k) = e^{ikx} + r(k) e^{-ikx}&&\quad\text{on } e_{\mathrm{in}},\\
	\label{AsympOut}
	&\psi(x,k) = t(k) e^{ikx} &&\quad\text{on } e_{\mathrm{out}}.
\end{alignat}
Here $r(k)$ and $t(k)$ denote the \textit{reflection} and \textit{transmission amplitudes}, respectively.
As follows from the analysis in Section~\ref{sec:single-resonator}, these amplitudes exist and are unique for every positive $k$.

The corresponding \textit{transmission} and \textit{reflection 
	probabilities} are defined by 
$$
    T(k)=|t(k)|^2, \qquad R(k)=|r(k)|^2.
$$    
The conservation of probability flux,
\begin{equation}\label{T+R=1}
	T(k)+R(k)=1,
\end{equation}
is a consequence of the self-adjointness of $H$~\cite{KostrykinSchrader1999}, 
which ensures that the probability current is conserved at each vertex.

Our goal is to analyse the transmission probability $T_N(k)$ as a function of the wave number~$k$ and the system parameters: the number~$N$ of resonators, the resonator length~$l$, the inter-resonator spacing~$h$, the vertex coupling type (I or II) and parameters~$\theta$, and the boundary conditions at the terminal vertices of the resonators. The main results---explicit closed-form formulas for $T_N(k)$, conditions for its periodicity or quasi-periodicity, quantitative gap-width asymptotics and their connection to the band structure of the infinite periodic system---are derived in Sections~\ref{sec:single-resonator}--\ref{sec:asymptotics}.

%% ==============================================

\section{Scattering by a single resonator}\label{sec:single-resonator}

%% ==============================================

\subsection{Transfer matrix} We consider first the case $N=1$, in which the quantum graph consists of a scattering channel with a single resonator attached to the line at the origin. To solve the scattering problem and determine the amplitudes $r(k)$ and $t(k)$, we construct a linear isomorphism in the two-dimensional space of solutions to  $\psi'' + k^2 \psi = 0$. This mapping relates the solution 
\begin{equation}\label{Psi1}
  \psi_1=a_1e^{ikx}+b_1e^{-ikx},\quad x\in(-\infty,0),
\end{equation}
on the incoming edge $e_{in}$ to the solution 
\begin{equation}\label{Psi2}
\psi_2=a_2e^{ikx}+b_2e^{-ikx}, \quad x\in(0,+\infty),
\end{equation}
on the outgoing edge $e_{out}$. The mapping is described by the transfer matrix $M(k)$ defined by
\begin{equation}\label{eq:transfer-matrix-equation}
    \begin{pmatrix}
    a_2\\
    b_2
    \end{pmatrix}
    =
    M(k)
    \begin{pmatrix}
    a_1\\
    b_1
    \end{pmatrix}.
\end{equation}
For real nonzero~$k$, this transfer matrix belongs to the special group $SU(1,1)$ and thus admits the representation~\cite{GriffithsSteinke}
\begin{equation}\label{eq:transfer-matrix-form}
	M(k) = \begin{pmatrix} w & z \\ z^* & w^* \end{pmatrix}
\end{equation}
with $\det M = |w|^2 - |z|^2 = 1$; here and hereafter, the star denotes complex conjugation. This structure is a consequence of Wronskian conservation for the Schr\"odinger equation with a self-adjoint interaction.

The scattering solution~\eqref{AsympIn}--\eqref{AsympOut} corresponds to $a_1=1$, $b_1=r(k)$, $a_2=t(k)$, $b_2=0$.
Substituting these values into~\eqref{eq:transfer-matrix-equation} and using~\eqref{eq:transfer-matrix-form} gives $t = w+zr$ and 
$0 = z^*+w^*r$, from which $r=-z^*/w^*$ and $t=1/w^*$, yielding~$M(k)$ in terms of scattering amplitudes as 
\begin{equation}\label{TransferMRepr}
	M(k) =
	\begin{pmatrix}
		1/t^*( k) & -r^*(k)/t^*(k)\\
		-r(k)/t(k) & 1/t( k)
	\end{pmatrix}.
\end{equation}
In particular, substituting~\eqref{TransferMRepr} into $|w|^2-|z|^2=1$ recovers the flux conservation relation~\eqref{T+R=1} for transmission and reflection probabilities.

The main goal in this section is to show that scattering by the edge resonator $e = (v,u)$ can effectively be modeled as scattering on the line subject to an energy-dependent point interaction at the vertex $v$. We will demonstrate that the resulting transfer matrix $M(k)$ possesses the $SU(1,1)$ structure \eqref{eq:transfer-matrix-form}, allowing us to utilize the representation \eqref{TransferMRepr} for the multi-resonator problem in Section~\ref{sec:local-periodic}.

We fix a nonzero energy $k^2\in\mathbb{R}$ and a parameter triple $\theta=(\theta_1,\theta_2,\theta_3)\in\mathbb{R}^3$. The conditions of types~I and~II are invariant under scaling of~$\theta$, so only its direction matters. Also, since the edge $e_{\mathrm{in}}$ is parameterised towards its vertex~$v$, the derivative of $\psi_1$ at this vertex in \eqref{eq:coupling-conditions-A} and \eqref{eq:coupling-conditions-B} satisfies $\psi_1'(v) = -\psi_1'(0)$. 

The construction of the transfer matrix relies on the explicit representation of the wave function on the resonator $e=(v,u)$. Upon identifying $e$ with the interval $[0,l]$, where $s=0$ corresponds to the vertex $v$, the component $\psi_3$ (corresponding to edge index~$3$ at vertex~$v$, as in Fig.~\ref{fig:condAB}) satisfying the Robin condition \eqref{eq:boundary-conditions-Robin} at the vertex $u$ is given by
\[
\psi_3(x,k) = \gamma \Bigl( \alpha \,\frac{\sin k(s-l)}{k} - \beta \cos k(s-l) \Bigr), \quad s \in [0,l],
\]
for some constant $\gamma \in \mathbb{C}$. The boundary values at the vertex $v$ are related via the Dirichlet-to-Neumann map $m$ as 
$$\psi_3'(0,k) = m(k)\psi_3(0,k),$$ 
where
\begin{equation}\label{Mk}
	m(k) = \frac{\psi'_3(0,k)}{\psi_3(0,k)} = k \frac{\beta k\tan{kl}-\alpha}{\alpha\tan{kl}+\beta k}.
\end{equation}
The function $m$ is a meromorphic function whose poles correspond to the Dirichlet eigenvalues~$k^2$ of the isolated resonator ($\psi_3(0,k)=0$). 

We next treat the vertex conditions of type I and type II separately. 

\subsection{Vertex conditions of type I}\label{ssec:typeI}  The coupling conditions \eqref{eq:coupling-conditions-A} can be written as
\begin{equation*}
   \theta_1 \psi_2(0) = \theta_2 \psi_1(0), \quad \theta_1 \psi_3(0) = \theta_3 \psi_1(0), \quad  -\theta_1 \psi'_1(0)+\theta_2 \psi'_2(0)+\theta_3 \psi'_3(0) = 0.
\end{equation*}
We first observe that the scattering is completely blocked whenever $\theta_1\theta_2=0$. Indeed, if $\theta_1=0$, then $\psi_1(0)=0$ and hence $r(k)=-1$ and $t(k)=0$ by~\eqref{AsympIn} and \eqref{T+R=1}. On the other hand, if $\theta_2=0$, then $\psi_2(0)=0$, yielding  $t(k)=0$ by~\eqref{AsympOut}.
In both cases, the transmission probability $T(k)$ vanishes identically, which means that the channel is completely blocked and no particle can pass through the vertex~$v$. Therefore, from now on we assume that $\theta_1\theta_2\neq0$.

Using the above coupling conditions and the Dirichlet-to-Neumann map~\eqref{Mk}, we get 
$
	\psi'_3(0) = m(k)\psi_3(0) = \theta_1^{-1}\theta_3\, m(k)\psi_1(0);
$
this allows one to eliminate $\psi_3$ completely and obtain a closed system involving only $\psi_1$ and $\psi_2$:
\begin{equation*}
       \psi_2(0) = \theta_2\theta_1^{-1} \psi_1(0),\qquad
       \psi'_2(0) = \theta_1\theta_2^{-1} \psi'_1(0) - \theta_1^{-1}\theta_2^{-1}\theta_3^2\, m(k)\psi_1(0).
\end{equation*}
These relations define an energy-dependent point interaction on the line.
Using the representations \eqref{Psi1} and \eqref{Psi2} for $\psi_1$ and $\psi_2$, these relations can be rewritten in matrix form as
\begin{equation*}
    \begin{pmatrix}
        1 & 1 \\
        ik & -ik
    \end{pmatrix}
    \begin{pmatrix}
        a_2 \\
        b_2
    \end{pmatrix}= \frac{1}{\theta_1\theta_2}
    \begin{pmatrix}
       \theta_2^2 & \theta_2^2 	 \\
        ik\theta_1^2 -  \theta_3^2\, m(k) &  -ik\theta_1^2 -  \theta_3^2\, m(k)
    \end{pmatrix}
    \begin{pmatrix}
        a_1 \\
        b_1
    \end{pmatrix};
\end{equation*}
solving it for $a_2$ and $b_2$ for $k \ne 0$, we then obtain the explicit form of the transfer matrix for a single resonator coupled by conditions of the first type:
\begin{equation}\label{eq:transfer-matrix-A-conditions}
    M(k) =\frac{1}{2\theta_1\theta_2}
    \begin{pmatrix}
          \phantom{-} \theta_1^2+ \theta_2^2+ ik^{-1}m(k)\theta_3^2 &-\theta_1^2 +\theta_2^2+i k^{-1}m(k)\theta_3^2\\
         -\theta_1^2 +\theta_2^2-i k^{-1}m(k)\theta_3^2 & \phantom{-}\theta_1^2+ \theta_2^2- ik^{-1}m(k)\theta_3^2
    \end{pmatrix}.
\end{equation}
This transfer matrix belongs to~$SU(1,1)$ for all real nonzero~$k$ that are not poles of~$m$, as one verifies directly from the 
explicit expressions for $w$ and $z$ above. Therefore, in view of \eqref{TransferMRepr}, the transmission amplitude is 
\begin{equation*}
  t(k)=\frac{2\theta_1\theta_2}{\theta_1^2+ \theta_2^2- ik^{-1}m(k)\theta_3^2}.
\end{equation*}
This formula remains valid even when $ \theta_1\theta_2= 0$,  although the transfer matrix  $M(k)$ is not defined in this case.

\subsection{Vertex conditions of type II}\label{ssec:typeII} We start with rewriting the couplings in the form
\begin{equation*}
   \theta_1 \psi_2'(0) = -\theta_2 \psi_1'(0), \quad \theta_1 \psi_3'(0) = -\theta_3 \psi_1'(0), \quad  \theta_1 \psi_1(0)+\theta_2 \psi_2(0)+\theta_3 \psi_3(0) = 0,
\end{equation*}
and again observe that the resonator completely blocks the channel when  $\theta_1\theta_2=0$. Indeed, if $\theta_1=0$, then $\psi_1'(0)=0$ and therefore $r(k)=1$ by \eqref{AsympIn}, while if $\theta_2=0$, then $\psi_2'(0)=0$ and hence $t(k)=0$. Therefore, we assume throughout that $\theta_1\theta_2\ne0$. 

Using the Dirichlet-to-Neumann map, we eliminate the resonator component $\psi_3(0)$ via $\psi_3(0)= \psi'_3(0)m(k)^{-1} = -\theta_1^{-1}\theta_3\, m(k)^{-1}\psi_1'(0)$ and obtain a closed system for $\psi_1$ and $\psi_2$, 
\begin{equation*} %\label{PointIntB}
\psi_2(0) = -\theta_1\theta_2^{-1} \psi_1(0) + \theta_1^{-1}\theta_2^{-1}\theta_3^2\, m(k)^{-1}\psi_1'(0),\qquad\psi_2'(0) = -\theta_1^{-1}\theta_2 \psi_1'(0).
\end{equation*}
Using the representation of $\psi_1$ and $\psi_2$ as in~\eqref{Psi1} and \eqref{Psi2}, we rewrite this system in terms of the coefficients $a_j$ and $b_j$, $j=1,2$, as
\begin{equation*}
    \begin{pmatrix}
       1 & 1 \\
        1 & -1
    \end{pmatrix}
    \begin{pmatrix}
        a_2 \\
        b_2
    \end{pmatrix}= \frac{-1}{\theta_1\theta_2}
    \begin{pmatrix}
                \theta_1^2  - ik m(k)^{-1}\theta_3^2 & \theta_1^2 +ik m(k)^{-1}\theta_3^2\\
                \theta_2^2  & -\theta_2^2
    \end{pmatrix}
    \begin{pmatrix}
        a_1 \\
        b_1
    \end{pmatrix}.
\end{equation*}
Solving this system for $a_2$ and $b_2$, we obtain the transfer matrix for a single resonator coupled by conditions of type~II:
\begin{equation}\label{eq:transfer-matrix-B-conditions}
    M(k) =-\frac{1}{2\theta_1\theta_2}
    \begin{pmatrix}
         \theta_1^2+\theta_2^2-\frac{ik\theta_3^2}{m(k)} &\theta_1^2 -\theta_2^2+\frac{ik\theta_3^2}{m(k)}\\
         \theta_1^2 -\theta_2^2-\frac{ik\theta_3^2}{m(k)} & \theta_1^2+\theta_2^2+\frac{ik\theta_3^2}{m(k)}
    \end{pmatrix}.
\end{equation}
By direct verification, this matrix belongs to~$SU(1,1)$ for all real~$k$ that are not zeros of~$m$; in view of \eqref{TransferMRepr}, for such~$k$ the transmission amplitude is explicitly given by
\begin{equation*}
  t(k)=-\frac{2\theta_1\theta_2}{\theta_1^2+ \theta_2^2+ik\theta_3^2/m(k)}.
\end{equation*}

\subsection{Transmission through one resonator} 
The results for both vertex conditions are conveniently stated in terms of the rescaled Dirichlet-to-Neumann coefficient
\begin{equation}\label{eq:mu}
	\mu(k)=\frac{m(k)}{k}=\frac{\beta k\tan{kl}-\alpha}{\beta k+\alpha\tan{kl}}.
\end{equation}

\begin{pro}
	Let $G$ be a quantum graph consisting of a scattering channel with a single resonator $e = (v,u)$ of length $l$ subject to the Robin boundary condition \eqref{eq:boundary-conditions-Robin} at the vertex~$u$. The transmission probability through the channel is given by
	\begin{equation}\label{eq:transmission-prob-A-conditions}
		T(k) = \frac{4\theta_1^2\theta_2^2}{(\theta_1^2+\theta_2^2)^2+\theta_3^4\mu^2(k)}
	\end{equation}
	for coupling conditions of type~I, and by
	\begin{equation}\label{eq:transmission-prob-B-conditions}
		T(k) = \frac{4\theta_1^2\theta_2^2}{(\theta_1^2+\theta_2^2)^2+\theta_3^4\,\mu^{-2}(k)}
	\end{equation}
	for coupling conditions of type~II. 
	
	In the low-energy limit $k \to 0$, if $\alpha \neq 0$, the transmission probability $T(k)$ tends to $0$ for the first type conditions and to $4\theta_1^2\theta_2^2/(\theta_1^2 + \theta_2^2)^2$ for the conditions of the second type. Conversely, if $\alpha = 0$, $T(k)$ tends to $4\theta_1^2\theta_2^2/(\theta_1^2 + \theta_2^2)^2$ for type~I and to $0$ for type~II.
\end{pro}

Indeed, the transmission formulas follow immediately from $T(k)=|t(k)|^2$ and the explicit expressions for $t(k)$ derived above. The low-energy limits follow from the behaviour of $\mu(k)$ as $k\to 0$: if $\alpha\ne 0$, then $\mu(k)\sim -\alpha/(k(\beta+\alpha l))\to\infty$, giving $T\to 0$ for type~I and $T\to 4\theta_1^2\theta_2^2/(\theta_1^2+\theta_2^2)^2$ for type~II; if $\alpha=0$, then $\mu(k)\sim kl\to 0$, reversing the roles.

We observe that the condition $\theta_3=0$ suppresses the interaction between the scattering channel and the resonator: the vertex conditions then force $\psi_3(v)=0$ for type~I (or $\psi_3'(v)=0$ for type~II). In this case 
$$
    T(k) = \frac{4\theta_1^2\theta_2^2}{(\theta_1^2+\theta_2^2)^2}
$$ 
by~\eqref{eq:transmission-prob-A-conditions} and~\eqref{eq:transmission-prob-B-conditions}, which is independent of~$k$ and ranges from perfect transmission $T= 1$ when $\theta_1=\theta_2$ to nearly zero as $\theta_1/\theta_2\to 0$ or $\theta_1/\theta_2\to\infty$. 
Any nonzero~$\theta_3$, however, introduces energy-dependent scattering and produces a band-pass filter structure  ---  alternating zones of near-perfect transmission and near-complete reflection  ---  as illustrated in Fig.~\ref{fig:N=1} and studied systematically in Sections~\ref{sec:properties} 
and~\ref{sec:asymptotics}.

\begin{figure}[h]
	\centering
	\includegraphics[scale = 0.32]{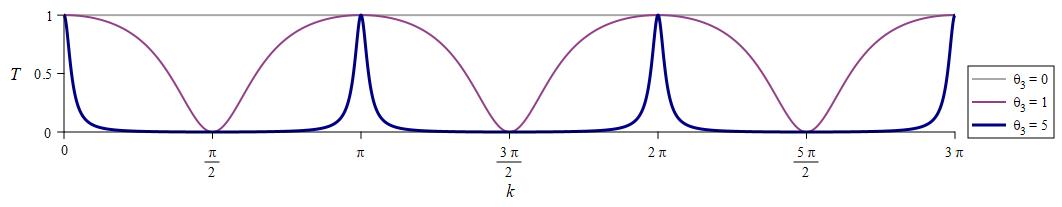}
	\caption{Plots of $T(k)$ for different values of $\theta_3$. Parameters: $\theta_1=\theta_2=1$, $l=h=1$,
		$\alpha=0$, and $\beta=1$.}
	\label{fig:N=1}
\end{figure}

\begin{rem}\rm
	Comparing~\eqref{eq:transfer-matrix-A-conditions} and~\eqref{eq:transfer-matrix-B-conditions}, one sees that the type-II transfer matrix~\eqref{eq:transfer-matrix-B-conditions} is obtained from the type-I one~\eqref{eq:transfer-matrix-A-conditions} by the substitution $m(k)\mapsto -m(k)^{-1}$. This observation underlies the duality between type-I and type-II conditions established in Property~\ref{prp:duality} of Section~\ref{sec:properties}.
\end{rem}

%% ======================================
\section{Scattering by multiple resonators} \label{sec:local-periodic}
%% ======================================

We now study the $N$-resonator system using the transfer matrix formalism~\cite{Mostafazadeh2020} and the explicit transfer matrices derived in Section~\ref{sec:single-resonator}. If the resonators are placed at equal distances and with the same vertex conditions, the system acquires a locally periodic structure along the channel. The construction of transfer matrices for such systems appears in many sources; see, for example, the detailed discussion in~\cite{GriffithsSteinke}.

First, we note that while the transmission probability is invariant under translations of the resonator along the channel, the transfer matrix is not. This reflects the fact that the transfer matrix must account for the phase accumulated by the wave as it propagates. To model a resonator attached at a point $x=h$, we represent the wave functions on either side of $x=h$ as linear combinations of $e^{ik(x-h)}$ and $e^{-ik(x-h)}$, effectively shifting the local origin to the attachment point. This basis shift yields the relation
\[
\begin{pmatrix}
	a_2 e^{ikh} \\
	b_2 e^{-ikh}
\end{pmatrix} = M(k)\begin{pmatrix}
	a_1 e^{ikh} \\
	b_1 e^{-ikh}
\end{pmatrix}
\]
for the coefficients $a_j$ and $b_j$ defined in~\eqref{Psi1} and \eqref{Psi2}. 
Denoting the transfer matrix at $x=h$ by $M(k,h)$ and introducing the free propagation matrix
\begin{equation*}
	P_h(k) = \begin{pmatrix} e^{ikh} & 0 \\ 0 & e^{-ikh} \end{pmatrix},
\end{equation*}
which corresponds to propagation over a distance $h$, we immediately conclude that 
\[
M(k,h)= P_h^{-1}(k)M(k)P_h(k).
\]

More generally, consider a system of $N$ resonators attached to the scattering channel at the points $x_j=(j-1)h$, $j=1,\dots,N$. The total transfer matrix $M_N(k)$ for the system is given by the ordered product
\begin{equation*}
	M_N(k) = M(k,x_{N})\cdots M(k,x_2)M(k,x_1),
\end{equation*}
where $M(k,x_j) = P_h(k)^{1-j}\, M(k)\, P_h(k)^{j-1}$ denotes the transfer matrix through the point $x_j$. 
Substituting this representation and rearranging the factors, we arrive at
\begin{multline*}
	M_N(k) = P_h^{1-N}(k)M(k)P_h^{N-1}(k)P_h^{2-N}(k)\cdots P_h^{-1}(k)M(k)P_h(k)M(k) \\
	= P_h^{-N}(k) \bigl(P_h(k)M(k)\bigr)^N = P_h^{-N}(k) Q(k,h)^N,
\end{multline*}
with $Q(k,h)=P_h(k)M(k)$ representing a single step of the scattering process, combining scattering by one resonator with free propagation over the distance~$h$.

Using the generic form~\eqref{eq:transfer-matrix-form} of the transfer matrix $M(k)$, we get
\begin{equation*}
	Q(k,h)=\begin{pmatrix}
		we^{ikh} & ze^{ikh} \\
		z^*e^{-ikh} & w^*e^{-ikh}
	\end{pmatrix},
\end{equation*}
which remains an $SU(1,1)$ matrix. The $N$-th power of $Q$ can be obtained in explicit terms by noting that, since $\det Q = 1$, the Cayley--Hamilton theorem yields the recurrence relation 
\begin{equation}\label{eq:CH}
	Q^{n+2} = 2\xi \, Q^{n+1} - Q^n,
\end{equation}
where $\xi$ denotes the half-trace of $Q$,  
\begin{equation*}
	\xi(k,h)=\frac{1}{2}\tr Q(k,h)=\Re(we^{ikh}).
\end{equation*}
This implies that the powers of $Q$ satisfy 
$Q^{n+2} = U_{n+1}(\xi)Q - U_n(\xi)I$, with $U_n$ denoting the Chebyshev polynomials of the second kind $U_n(\xi)$ defined by the initial values and the recursive relation 
\begin{equation}\label{eq:Chebyshev-identity}
	U_0(\xi)=1, \qquad U_1(\xi)=2\xi, \qquad U_{n+2}(\xi)=2\xi\,U_{n+1}(\xi)-U_{n}(\xi),
\end{equation}	
as explained in detail, e.g., in~\cite{GriffithsSteinke}.
Ultimately, this results in the explicit formula for the transfer matrix, 
\begin{equation*}
	M_N(k)=P_h^{-N}(k)\left(U_{N-1}\big(\xi(k,h)\big)\,Q(k,h) - U_{N-2}\big(\xi(k,h)\big)\, I\right).
\end{equation*}

By construction, $M_N$ belongs to $SU(1,1)$ and thus is related to the reflection and transmission coefficients of the whole $N$-resonator system via~\eqref{TransferMRepr}. In particular, by identifying the $(2,1)$-entry of $M_N$ as \(e^{ikNh} U_{N-1}(\xi) [Q]_{21}\), where $[Q]_{21}=z^*e^{-ikh}$ from the explicit form of $Q(k,h)$ above, we find that
\begin{equation*}
	\frac{r_N(k)}{t_N(k)}=-U_{N-1}(\xi)z^*(k)e^{ik(N-1)h}.
\end{equation*}
Taking the absolute squared value and using~\eqref{T+R=1}, we obtain
\begin{equation*}
	\frac{1-T_N(k)}{T_N(k)}=|z(k)|^2 U^2_{N-1}(\xi),
\end{equation*}
and hence the final formula, 
\begin{equation}\label{TNzxi}
	T_N(k)=\frac{1}{1+|z(k)|^2\,U_{N-1}^2\big(\xi(k)\big)}.
\end{equation}

We note that the general formula \eqref{TNzxi} for the transmission probability across locally periodic point interactions is well-known in the literature (see, e.g., \cite{GriffithsSteinke}). However, its practical application requires the explicit construction of the transfer matrix $M(k)$ for the specific basic cell—i.e., determining the functions $w(k)$ and $z(k)$---which we carried out for our resonator models in Section 3. Substituting those results into \eqref{TNzxi} yields the main analytical result for the multi-resonator chain.

The following theorem summarizes the transmission probabilities for the $N$-resonator model.

\begin{thm}\label{thm:TN}
	Let $G$ be a metric graph consisting of a scattering channel with $N$ identical resonators of length~$l$, attached at points $x_j=(j-1)h$, $j=1,\dots,N$. 
	Define the Hamiltonian on~$G$ as in Section~\ref{sec:problem-formulation}, with Robin boundary condition~\eqref{eq:boundary-conditions-Robin} 
	at vertices~$u_j$ and identical vertex conditions at~$v_j$, of type~I or type~II, with the same parameter vector~$\theta$. Let $\mu$ be given by~\eqref{eq:mu}.
	
	(i) Suppose that the resonators are coupled to the channel by vertex conditions of type~I and $\theta_1\theta_2\neq0$. Then the transmission probability is given by
	\begin{equation}\label{eq:TN-A}
		T_N(k)=\frac{1}{1+\dfrac{(\theta_1^2-\theta_2^2)^2+\mu^2(k)\theta_3^4}{4\theta_1^2\theta_2^2}\;\,
			U_{N-1}^2\!\left(\dfrac{(\theta_1^2+\theta_2^2)\cos kh-\mu(k)\theta_3^2\sin kh}{2\theta_1\theta_2}\right)}.
	\end{equation}
	
	\smallskip
	
	\noindent
	(ii) Suppose that the resonators are coupled to the channel by vertex conditions of type~II and $\theta_1\theta_2\neq0$. Then the transmission probability is given by
	\begin{equation}\label{eq:TN-B}
		T_N(k)=\frac{1}{1+\dfrac{(\theta_1^2-\theta_2^2)^2+\mu^{-2}(k)\theta_3^4}{4\theta_1^2\theta_2^2}\;\,
			U_{N-1}^2\!\left(\dfrac{(\theta_1^2+\theta_2^2)\cos kh+\mu^{-1}(k)\theta_3^2\sin kh}{2\theta_1\theta_2}\right)}.
	\end{equation}
	
	\smallskip
	
	\noindent
	(iii) If $\theta_1\theta_2=0$, then $T_N(k)=0$ for all nonzero $k\in\Real$. In this case, the scattering channel is completely blocked by the resonators, and no tunneling occurs.
\end{thm}

Finally, we observe that for a single resonator, we have $U_0(\xi)=1$, and therefore 
\begin{equation*}
	T_1(k)=\frac{1}{1+|z(k)|^2}.
\end{equation*}
Identifying $|z(k)|^2$ from~\eqref{eq:transfer-matrix-A-conditions} 
and~\eqref{eq:transfer-matrix-B-conditions}, we immediately recover the transmission probabilities \eqref{eq:transmission-prob-A-conditions} and \eqref{eq:transmission-prob-B-conditions} obtained earlier for the single-resonator case.

%% ==========================================
\section{Properties of the transmission probability}\label{sec:properties}
%% ==========================================

In this section, we analyze the structural properties of $T_N$ derived in Theorem~\ref{thm:TN}. We begin by listing invariance and symmetry properties that follow immediately from the explicit formula and reduce the effective parameter space. We then establish a duality between the two types of vertex and boundary conditions. In Subsection~\ref{ssec:phase-rep}, a special representation of $T_N$ in terms of accumulated phases is suggested, which enables derivation of several analytical results on large-energy behavior: periodicity or quasi-periodicity of $T_N$, persistence of transmission and reflection resonances at arbitrarily high energies, quantitative behaviour of $T_N$ in the transmission suppression zones, and asymptotic proximity of $T_N$ for Robin and Neumann conditions. 

We observe that in the case $\theta_3=0$, the system reduces to a $\delta'$-comb, i.e., an array of energy-independent point interactions
\[
	\psi(x_j{+}0)-\nu\,\psi(x_j{-}0)=0,\quad
	\nu\,\psi'(x_j{+}0)-\psi'(x_j{-}0)=0,
	\quad\nu=\theta_2/\theta_1,
\]
which was studied in detail in~\cite{GolovatyHrynivLavrynenko2025}. 
The present paper extends the $\delta'$-comb model by including resonators through non-trivial vertex conditions with~$\theta_3\ne0$. Also, as explained in Subsections~\ref{ssec:typeI} and \ref{ssec:typeII}, no transmission occurs through the vertices~$v_j$ when $\theta_1\theta_2=0$. Therefore, throughout this and the next section, we make a standing assumption that $\theta_1\theta_2\theta_3\ne0$. 

%%  ---  Invariance properties  --- 

\subsection{Symmetry and invariance}
In what follows, we omit the subscript $N$ in the transmission probability $T_N$ if the discussed property does not depend on the number~$N$ of resonators.

\begin{prp}[Scale invariance of coupling parameters]\label{prp:scale-invariance}\normalfont
	The transmission probability $T$ is 
	homogeneous of degree zero in $(\theta_1,\theta_2,\theta_3)$:
	\[
	T(k;\lambda\theta_1,\lambda\theta_2,\lambda\theta_3)
	= T(k;\theta_1,\theta_2,\theta_3)
	\quad\text{for all }\lambda\in\mathbb{R}\setminus\{0\}.
	\]
\end{prp}

This property is immediate from~\eqref{eq:TN-A}--\eqref{eq:TN-B}, since $\theta$ appears only as ratios $\theta^2_j/(\theta_1\theta_2)$.

\begin{prp}[Evenness in coupling parameters]\label{prp:evenness}\normalfont
	The transmission probability is an even function of each $\theta_j$ 
	separately:
	\[
	T(k;\pm\theta_1,\pm\theta_2,\pm\theta_3)
	= T(k;\theta_1,\theta_2,\theta_3).
	\]
\end{prp}
This follows from~\eqref{eq:TN-A}--\eqref{eq:TN-B} and the parity 
	relation $U_n(-x)=(-1)^n U_n(x)$.
    
\begin{prp}[Spatial parity]\label{prp:lr-symmetry}\normalfont
	The transmission probability is invariant under interchange of 
	$\theta_1$ and $\theta_2$:
	\[
	T(k;\theta_1,\theta_2,\theta_3)=T(k;\theta_2,\theta_1,\theta_3).
	\]
\end{prp}

This is immediate from~\eqref{eq:TN-A}--\eqref{eq:TN-B} and expresses the parity symmetry of the model under
the spatial reflection that interchanges the channels $e_{\mathrm{in}}$ and $e_{\mathrm{out}}$.

\begin{prp}[Geometric scaling]\label{prp:geometric-scaling}\normalfont
	The transmission probability satisfies
	\[
	T(k;h,l) = T(\lambda k;\, h/\lambda,\, l/\lambda)
	\quad\text{for all }\lambda>0.
	\]
\end{prp}

This follows from the fact that $\mu(k)$ and $\cos kh$ depend on $k$, $l$, $h$ only through the products $kl$ and $kh$, so the formula~\eqref{eq:TN-A} (and~\eqref{eq:TN-B}) is invariant under $k\mapsto\lambda k$, $l\mapsto l/\lambda$, $h\mapsto h/\lambda$. In particular, choosing $\lambda=h$ shows that $h$ may be set to~$1$ without loss of generality.

Together, Properties~\ref{prp:scale-invariance}--\ref{prp:geometric-scaling}
reduce the parameter space substantially. The vector~$\theta$ may be 
restricted to the fundamental domain
\[
\bigl\{\theta\in S^2 \colon 
\theta_1\ge\theta_2 > 0,\;\theta_3 > 0\bigr\},
\]
and $h$ may be normalized to~$1$, leaving $kl$ and $l/h$ as the only geometric parameters.

We next point out a symmetry between the two types of vertex 
conditions at~$v_j$ and the two standard boundary conditions at the vertices~$u_j$. In the following, we use superscripts to 
distinguish these cases: $T^{\mathrm{I,N}}$ and $T^{\mathrm{I,D}}$ denote the 
transmission probability for type-I vertex conditions with Neumann 
 and Dirichlet conditions at~$u_j$, 
respectively, and similarly for type-II.

\begin{prp}[Neumann--Dirichlet duality]\label{prp:duality}\normalfont
	The transmission probability is invariant under simultaneous 
	interchange of vertex coupling type and endpoint boundary condition:
	\begin{equation}\label{eq:duality}
		T^{\mathrm{I,N}}(k) = T^{\mathrm{II,D}}(k), \qquad T^{\mathrm{I,D}}(k) = T^{\mathrm{II,N}}(k).
	\end{equation}
\end{prp}

The key observation is that the formula for type-I conditions becomes 
that for type-II upon replacing $\mu$ by $-\mu^{-1}$,
cf.~\eqref{eq:TN-A}--\eqref{eq:TN-B}. Switching from Neumann to 
Dirichlet boundary conditions performs exactly this substitution (cf.~\eqref{eq:mu}):
\[
\mu^{(N)}(k) = \tan kl 
\quad\longmapsto\quad 
\mu^{(D)}(k) = -\cot kl = -\bigl(\mu^{(N)}(k)\bigr)^{-1},
\]
which yields~\eqref{eq:duality}.

In the transmission-line analogy,  $\mu(k)$ plays the role of a local admittance and $\mu(k)^{-1}$
that of an impedance; Property~\ref{prp:duality} then states that switching between the two descriptions is equivalent to swapping the vertex-coupling type. A direct consequence is that the roles of zeros and poles of $\mu$ are interchanged: for type-I, zeros of $\mu$ produce perfect transmission and poles produce 
complete reflection; for type-II, these roles are reversed. In view of Property~\ref{prp:duality}, it suffices to analyze type-I conditions with Neumann or Dirichlet endpoints; the type-II cases follow by duality. As we show in Property~\ref{prp:Robin-Neumann}, for type-I vertex and Robin conditions, the transmission probability approaches, at large energies, the one for Neumann conditions.

%% ---  ---  ---  ---  ---  ---  ---  ---  ---  ---  ---  ---  ---  ---  ---  --- --
\subsection{Phase representation of the transmission
	probability}
\label{ssec:phase-rep}
%% ---  ---  ---  ---  ---  ---  ---  ---  ---  ---  ---  ---  ---  ---  ---  --- --

As explained in the previous subsection, the study of $T$ can be restricted to the case $\theta_j>0$, and this will be assumed throughout Sections~\ref{ssec:phase-rep} and~\ref{ssec:large-energy}. We also fix the type-I vertex conditions; the type-II conditions are treated similarly. We set
\begin{equation}\label{eq:AB}
	A = \frac{\theta_1^2+\theta_2^2}{2\theta_1\theta_2}\ge1,
	\qquad
	B = \frac{\theta_3^2}{2\theta_1\theta_2} > 0;
\end{equation}
then the entries of the transfer
matrix~\eqref{eq:transfer-matrix-A-conditions} satisfy
\begin{equation}\label{eq:wz-AB}
	w(k) = A+iB\mu(k),
	\qquad
	|z(k)|^2 = (A^2-1)+B^2\mu^2(k),
\end{equation}
and the half-trace is
$\xi(k)=\Re(we^{ikh})=A\cos kh-B\mu(k)\sin kh$.

\medskip
\noindent\textit{Phase variables.}
We introduce two phases that give the representation of $T$ its natural structure.
The \emph{propagation phase} $\psi(k)=kh$ is the de~Broglie phase accumulated by a free particle of
wavenumber~$k$ traveling from one attachment point
to the next.
The \emph{resonator phase} $\phi(k)$ is derived from the
relation $\mu(k)=\tan\phi(k)$ and equals
\begin{equation}\label{eq:phases}
	\phi(k) = kl-\delta(k),
	\qquad
	\delta(k) = \arctan\frac{\alpha}{\beta k},
\end{equation}
with $\delta(k)\equiv0$ for Neumann~($\alpha=0$),
$\delta(k)\equiv\tfrac\pi2$ for Dirichlet~($\beta=0$),
and $\delta(k)=O(k^{-1})$ as $k\to\infty$ for
Robin conditions~($\alpha\beta\ne0$). Here, $\delta(k)$ gives the phase shift of the free wave at the vertex~$u_j$ caused by the respective condition. 

To derive $\delta$ in the case $\alpha\beta\ne0$, we divide the numerator and
denominator of~\eqref{eq:mu} by $\beta k$ and use
$\tan\delta(k)=\alpha/(\beta k)$ to arrive at
\[
\mu(k) = \frac{\beta k \tan kl - \alpha}{\beta k + \alpha \tan kl} = \frac{\tan kl - \tan\delta(k)}{1+\tan kl\tan\delta(k)}
= \tan(kl-\delta(k)) = \tan\phi(k),
\]
by the tangent subtraction formula. For Dirichlet conditions ($\beta=0$), the equality
$\mu(k)=-\cot kl=\tan(kl-\tfrac\pi2)=\tan\phi(k)$ is verified directly.

Physically, $\phi(k)=kl-\delta(k)$ is the phase accumulated by the wave on a single pass along the
resonator, $kl$, corrected by the reflection phase shift $\delta$ at the pendant vertex~$u_j$. Complete
suppression of transmission occurs when the isolated resonator, closed by a Dirichlet condition at the
attachment vertex~$v_j$, has an eigenstate at energy~$k^2$; the corresponding quantization condition,
of Bohr--Sommerfeld type, reads
\[
  \phi(k)=\tfrac\pi2+n\pi, \quad n\in\mathbb{Z},
\]
for all three boundary conditions at~$u_j$, and
gives the poles of~$\mu(k)=\tan\phi(k)$.

\begin{lem}\label{lem:phi}
	The resonator phase $\phi(k)=kl-\delta(k)$
	is continuous and satisfies $\phi(k)\to+\infty$
	as $k\to+\infty$.
	Define
	\begin{equation}\label{eq:kappa1-def}
		\kappa_1 = \inf\bigl\{k>0:\phi(k)=\tfrac\pi2\bigr\},
	\end{equation}
	with the convention $\kappa_1=0$ if no such $k$
	exists. Then $\phi'(k)>0$ for all $k>\kappa_1$,
	so that $\phi$ is strictly increasing on~$(\kappa_1,\infty)$.
\end{lem}

\begin{proof}
	Continuity and $\phi\to+\infty$ are clear from
	the explicit formula $\phi(k)=kl-\delta(k)$ with
	$\delta(k)\to0$. For Neumann ($\alpha=0$) and Dirichlet ($\beta=0$)
	conditions, $\phi(k) = kl$ and $\phi(k) = kl - \pi/2$, and the remaining claims are straightforward.
	
	For Robin conditions ($\alpha\beta\ne0$),
	\begin{equation}\label{eq:phi-prime}
		\phi'(k) = l+\frac{\alpha\beta}{\alpha^2+\beta^2k^2},
	\end{equation}
	and two cases $\alpha\beta>0$ and $\alpha\beta<0$ should be considered separately. 

    \smallskip 
	\emph{Case~1: $\alpha\beta>0$.}
	Then $\phi'(k)\ge l>0$ for all $k\ge0$, and $\phi(0^+)=-\pi/2<\pi/2$, so $\phi(k)<\pi/2$
	initially and $\phi$ is strictly increasing for $k>0$. The value $\kappa_1>0$ is the unique solution of $\phi(k) = \pi/2$. 

    \smallskip 
	\emph{Case~2: $\alpha\beta<0$.}
	Then $\phi(0^+)=\pi/2$, $\phi'$ increases in~$k$, and the conclusions depend on $\phi'(0) = l + \beta/\alpha$. 
	
	\emph{Case~2a: $\phi'(0^+)=l-|\beta/\alpha|\ge0$.}
	Since $\phi'$ is strictly increasing and non-negative at $0^+$, $\phi'(k)>0$ for all
	$k>0$, and $\phi$ is strictly increasing there; hence $\phi(k)>\pi/2$ for all $k>0$
	and we set $\kappa_1=0$. 
	
	\emph{Case~2b: $\phi'(0^+)=l-|\beta/\alpha|<0$.}
	Then $\phi$ decreases initially below $\pi/2$; since $\phi(k)\to+\infty$, the intermediate value theorem gives a smallest $\kappa_1>0$ with $\phi(\kappa_1)=\pi/2$. Then $\cot(\kappa_1 l)=-\tan\delta(\kappa_1)$,
	i.e., $\alpha=-\beta\kappa_1\cot(\kappa_1 l)$. Substituting into~\eqref{eq:phi-prime} yields
	\[
	   \kappa_1\phi'(\kappa_1) 	= \kappa_1l - \cos\kappa_1l\sin\kappa_1l= \tfrac12\bigl(2\kappa_1 l -\sin(2\kappa_1 l)\bigr)>0,
	\]
	as $\theta>\sin\theta$ for all $\theta>0$.
	Since $\phi'$ is strictly increasing in $k$,
	$\phi'(k)>\phi'(\kappa_1)>0$ for all
	$k>\kappa_1$.
	
	In all cases, $\phi' >0$ on $(\kappa_1,+\infty)$ and thus $\phi$ is strictly increasing there.	
\end{proof}

\begin{cor}\label{cor:mu}
  The poles of $\mu$ in $[\kappa_1,+\infty)$ form an unbounded strictly increasing sequence $\kappa_1<\kappa_2<\cdots$. Each $\kappa_n$ satisfies $\phi(\kappa_n) = (2n-1)\pi/2$, $n\in\mathbb{N}$; moreover, $\mu$ is strictly increasing on each interval $(\kappa_n,\kappa_{n+1})$.
\end{cor}
\begin{proof}
  By construction, the poles of~$\mu=\tan\phi$ on $\Real_+$ are exactly those $k$ for which~$\phi$ takes values $(2n-1)\pi/2$. Since $\phi$ is strictly increasing on $[\kappa_1,+\infty)$ with $\phi\to+\infty$, it takes each value $(2n-1)\pi/2$, $n\in\mathbb{N}$, exactly once there, thus uniquely defining a strictly increasing
  unbounded sequence $\kappa_1<\kappa_2<\cdots$ of poles of~$\mu$. Monotonicity of $\mu$ between poles follows from $\phi'>0$ and the chain rule.
\end{proof}

The behaviour of the half-trace 
\[
    \xi(k)=A\cos kh - B\mu(k)\sin kh = A \cos\psi(k) - B \mu(k)\sin\psi(k)
\]
near each pole $\kappa_n$ of $\mu$
depends on whether $\kappa_n$ is simultaneously a zero of $\sin \psi(k)$: in the generic case $\xi$ inherits
a pole from $\mu$, while in the exceptional case the singularity is removable.

\begin{lem}%[Behaviour of $\xi$ near poles of $\mu$]
\label{lem:xi-asymptotics}
At each pole $\kappa_n$ of $\mu$, the following holds:
\begin{enumerate}
  \item[\textup{(i)}]
  if $\sin(\kappa_n h)\ne0$, then $\kappa_n$
  is a pole of $\xi$;
  \item[\textup{(ii)}]
  if $\kappa_n h=m\pi$ for some $m\in\mathbb{N}$, then
  \begin{equation}\label{eq:xi-limit}
    \xi(k)
    = (-1)^m\!\Bigl(A+\frac{Bh}{\phi'(\kappa_n)}\Bigr) + O(k-\kappa_n)
    \quad\text{as }k\to\kappa_n;
  \end{equation}
  in particular, the limiting value $\xi(\kappa_n)$ satisfies $|\xi(\kappa_n)|>1$.
\end{enumerate}
\end{lem}

\begin{proof}
Only part~\textup{(ii)} requires justification.
Writing 
\[
    \frac{\sin\psi(k)}{\cos\phi(k)} = \frac{\sin\psi(k)/(k-\kappa_n)}{\cos\phi(k)/(k-\kappa_n)},
\]
we represent $\mu(k)\sin kh = \tan \phi(k)\sin \psi(k) = \sin \phi(k) \sin \psi(k)/\cos\phi(k)$ as a continuously differentiable function in a neighbourhood of $\kappa_n$. Since
    \[
        \lim_{k\to \kappa_n} \frac{\sin\phi(k)\sin\psi(k)/(k-\kappa_n)}{\cos\phi(k)/(k-\kappa_n)} 
            = - \frac{\sin\phi(\kappa_n)\psi'(\kappa_n) \cos\psi(\kappa_n)}{\phi'(\kappa_n)\sin\phi(\kappa_n)}
            = - \frac{h(-1)^m}{\phi'(\kappa_n)},
    \]
we get, as $k\to\kappa_n$,  
    \[
        \xi(k) = A\cos\psi(k)-B\tan\phi(k)\sin\psi(k) \to (-1)^m \Bigl(A + \frac{Bh}{\phi'(\kappa_n)}\Bigr) := \xi(\kappa_n).
    \]
With this continuously differentiable representation of~$\xi$, the difference $\xi(k)-\xi(\kappa_n)$ is $O(k-\kappa_n)$, giving the stated remainder in~\eqref{eq:xi-limit}. The lower bound $|\xi(\kappa_n)|=A+Bh/\phi'(\kappa_n)>A\ge1$ is immediate.
\end{proof}

\begin{pro}[Central representation]\label{pro:central}
	Under the assumptions of Theorem~\ref{thm:TN}\,\textup{(i)},
	with $\psi(k)=kh$ and $\phi(k)$ defined by~\eqref{eq:phases},
	\begin{equation}\label{eq:TN-cF}
		T(k) = \cF\bigl(\phi(k),\psi(k)\bigr),
	\end{equation}
	where
	\begin{equation}\label{eq:cF}
		\cF(\phi,\psi)
		=\frac{\cos^2\!\phi}
		{\cos^2\!\phi
			+\bigl[(A^2-1)\cos^2\!\phi+B^2\sin^2\!\phi\bigr]
			\,U_{N-1}^2\!\bigl(A\cos\psi-B\tan\phi\sin\psi\bigr)}.
	\end{equation}
	The function~$\cF$ is $\pi$-periodic in each argument,
	takes values in $[0,1]$, and is smooth on
	$\{(\phi,\psi):\phi\ne\pi n -\tfrac\pi2,\,n\in\mathbb{Z}\}$.
	Setting $\cF(\pi n -\tfrac\pi2,\psi)=0$ for
	$\sin\psi\ne0$ extends~$\cF$ continuously to those
	lines away from the lattice points~$(\phi,\psi)\in(\pi\mathbb{Z}-\tfrac\pi2) \times \pi \mathbb{Z}$.
\end{pro}

\begin{proof}
	Substituting $\mu=\tan\phi$ into~\eqref{TNzxi} and~\eqref{eq:wz-AB} gives
	\[
	T(k)
	= \frac{1}{1+\bigl[(A^2-1)+B^2\tan^2\!\phi(k)\bigr]
		U_{N-1}^2(A\cos\psi(k)-B\tan\phi(k)\sin\psi(k))};
	\]
    multiplying next the numerator and denominator by $\cos^2\phi(k)$ yields~\eqref{eq:TN-cF}--\eqref{eq:cF}.
	
	Since $\cF$ depends on $\phi$ through $\tan\phi$, $\cos^2\phi$, $\sin^2\phi$, it is $\pi$-periodic in~$\phi$. Replacing $\psi\mapsto\psi+\pi$ changes sign in both	$\cos\psi$ and $\sin\psi$, hence in the argument of~$U_{N-1}$; since $U_{N-1}^2$ is an even function, $\cF$ is unchanged.
	
	The fact that $\cF\ge 0$ is clear; $\cF \le 1$ because the denominator is at least $\cos^2\phi$: indeed, in view of $A\ge1$ and $B>0$, the bracket $(A^2-1)\cos^2\phi+B^2\sin^2\phi$ is non-negative.

    Continuity is only questionable at the poles of $\tan\phi$ and zeros of the denominator. Set
	\begin{equation}\label{eq:tildeU}
		\widetilde{U}_{N-1}(\phi,\psi)
		:= U_{N-1}(A\cos\psi-B\tan\phi\sin\psi)
		\cdot(\cos\phi)^{N-1};
	\end{equation}
	expanding $U_{N-1}$ as a polynomial of degree $N-1$
	with leading coefficient $2^{N-1}$ gives
	\begin{equation}\label{eq:tildeU-trig}
		\widetilde{U}_{N-1}(\phi,\psi)
		= \sum_{j=0}^{N-1}a_j (A\cos\phi\cos\psi-B\sin\phi\sin\psi)^j (\cos \phi)^{N-1-j},
	\end{equation}
	a trigonometric polynomial in $(\phi,\psi)$, which is uniformly bounded and uniformly continuous
	on $\mathbb{R}^2$. Substituting
	$\widetilde{U}_{N-1}^2=(\cos\phi)^{2(N-1)}\cdot
	U_{N-1}^2(\xi)$ into~\eqref{eq:cF} yields
	\begin{equation}\label{eq:cF-tildeU}
		\cF(\phi,\psi)
		= \frac{\cos^{2N}\!\phi}
		{\cos^{2N}\!\phi
			+\bigl[(A^2-1)\cos^2\!\phi+B^2\sin^2\!\phi\bigr]
			\widetilde{U}_{N-1}^2(\phi,\psi)},
	\end{equation}
	in which every factor is continuous on $\mathbb{R}^2$.
	At $\phi=\pi n - \tfrac\pi2$, the denominator
	equals $B^2\widetilde{U}_{N-1}^2(\pi n - \tfrac\pi2,\psi)
	=4^{N-1}B^{2N}(\sin\psi)^{2(N-1)}$,
	which is strictly positive for $\sin\psi\ne0$.
	Hence the ratio~\eqref{eq:cF-tildeU} extends
	continuously with value~$0$ at all such points.
\end{proof}

 \begin{rem} Denote by
  \begin{equation}\label{eq:Lambda}
    \Lambda
    = \Bigl(\pi\mathbb{Z}-\tfrac\pi2\Bigr)
       \times \pi\mathbb{Z}
    \subset\mathbb{R}^2
  \end{equation}
  the lattice of points at which $\cos\phi=0$ and
  $\sin\psi=0$ simultaneously. By Proposition~\ref{pro:central}, the function~$\cF$
  is continuous on $\mathbb{R}^2\setminus\Lambda$; as follows from Remark~\ref{rem:corners} below, $\cF$ is discontinuous at every point of~$\Lambda$ since the limit therein depends on the direction.
  \end{rem}

\begin{pro}[Zeros and ones of $\cF$]
	\label{pro:zeros-ones}
	Let $(\phi,\psi)\notin \Lambda$.
	\begin{enumerate}
		\item[\textup{(i)}]
		$\cF(\phi,\psi)=0$ if and only if $\phi=\pi n -\tfrac\pi2$ for some $n\in\mathbb{Z}$ .
		\item[\textup{(ii)}]
		$\cF(\phi,\psi)=1$ if 
		\begin{equation}\label{eq:cF=1}
			U_{N-1}(A\cos\psi-B\tan\phi\sin\psi)=0.
		\end{equation}
        Extra solutions of $\cF(\phi,\psi)=1$ are only possible when $A=1$ (i.e., when $\theta_1=\theta_2$), in which case they coincide with $\pi \mathbb{Z} \times \Real$.
	\end{enumerate}
\end{pro}

\begin{proof}
	\emph{Part~(i).}
	The numerator of~\eqref{eq:cF-tildeU} vanishes if and only if
	$\cos\phi=0$, i.e., if and only if $\phi=\pi n-\tfrac\pi2$.
	Outside the excluded lattice points, the denominator of~\eqref{eq:cF-tildeU}
	equals $B^2\widetilde{U}_{N-1}^2(\pi/2,\psi)
	=4^{N-1}B^{2N}(\sin\psi)^{2(N-1)}>0$, so the ratio is indeed zero.
	
	\emph{Part~(ii).}
	Since $\phi\ne\pi n - \tfrac\pi2$, we have $\cos\phi\ne0$, and thus $\cF(\phi,\psi)=1$ if and only if
    \[
        \bigl[(A^2-1)\cos^2\!\phi+B^2\sin^2\!\phi\bigr]
			{U}_{N-1}^2(A\cos\psi-B\tan\phi\sin\psi) = 0.
    \]
    The expression in the square bracket is always positive if $A>1$; for $A=1$, it vanishes if $\sin\phi = 0$ regardless of $\psi$. 
\end{proof}

\begin{rem}\label{rem:corners}\rm
	At the lattice points $(\phi_0,\psi_0)\in \Lambda$,
	the limiting value of $\cF$ depends on the direction of approach. For example, with $N=2$ and $A=B=1$, the path $\phi=\tfrac{\pi}2-t$, $\psi=t$ satisfies $\widetilde{U}_{1}(\tfrac{\pi}2-t,t) \equiv 0$ and $\cF(\tfrac{\pi}2-t,t)\equiv 1$ for $t\in(0,\tfrac\pi2)$, while $\cF(\pi/2,t) \equiv 0$ for such $t$, resulting in different limits as $t\to0^+$.
	
	However, along the physical path $(\phi(k),\psi(k))$ through $\Lambda$, the limiting value of~$\cF$ is always zero. Indeed, assume that $(\phi(k_0),\psi(k_0))\in \Lambda$, i.e., that $k_0 = \kappa_n$ and $\kappa_nh = \pi m$ for some $n,m \in \mathbb{N}$. By Lemma~\ref{lem:xi-asymptotics}(ii), 
    the limiting value $\xi(\kappa_n)=\lim_{k\to\kappa_n}\xi(k)$ is outside~$(-1,1)$; thus $U_{N-1}(\xi(\kappa_n))\ne0$, and the denominator in~\eqref{eq:cF} tends to the nonzero value $B^2U^2_{N-1}(\xi(\kappa_n))$. This yields $T(\kappa_n) = \lim_{k\to\kappa_n}\cF(\phi(k),\psi(k)) = 0$, consistently with
	Proposition~\ref{pro:zeros-ones}(i) and Theorem~\ref{thm:high-energy}(i) below.
\end{rem}

%% ---  ---  ---  ---  ---  ---  ---  ---  ---  ---  ---  ---  ---  ---  ---  --- --
\subsection{Large-energy behaviour of~$T$}
\label{ssec:large-energy}
%% ---  ---  ---  ---  ---  ---  ---  ---  ---  ---  ---  ---  ---  ---  ---  --- --

We now derive the main analytic consequences of the
representation $T=\cF(\phi,\psi)$
established in Proposition~\ref{pro:central}.

%%  ---  Periodicity  --- 

\begin{prp}[Periodicity and quasi-periodicity]
	\label{prp:periodicity}\normalfont
	For Neumann or Dirichlet conditions at pendant vertices, the transmission probability~$T$ is	\begin{itemize}
		\item[(i)] periodic in~$k$ if $l/h\in\mathbb{Q}$;
		\item[(ii)] two-frequency quasi-periodic in~$k$ if $l/h\notin\mathbb{Q}$.
	\end{itemize}
\end{prp}

\begin{proof}
    For Neumann or Dirichlet conditions, $\delta$ is constant and, by~\eqref{eq:TN-cF} and~\eqref{eq:phases},  
	\begin{equation}\label{eq:TN-two-phases}
		T(k) = \cF(kl-\delta,\,kh)
	\end{equation}
	is a function of $k$ through the two phases $kh$ and $kl$ alone. 
    Since $\cF$ is $\pi$-periodic in each argument by Proposition~\ref{pro:central}, $T(k+P)=T(k)$ for all $k$ if and only if $Pl\in\pi\mathbb{Z}$ and $Ph\in\pi\mathbb{Z}$ simultaneously, which holds for some
	$P>0$ if and only if $l/h\in\mathbb{Q}$.
	
	When $l/h\notin\mathbb{Q}$, the pair $(kh\bmod\pi,\,kl\bmod\pi)$ traces a dense orbit in $[0,\pi)^2$ by Weyl's equidistribution theorem, so $T$ is quasi-periodic but not periodic.
\end{proof}

%%  ---  High-energy resonances  --- 
We now establish the persistence of anti-resonances
and transmission resonances at arbitrarily high energies.

\begin{thm}%[Transmission resonances and anti-resonances]
	\label{thm:high-energy}\normalfont
	Under the standing assumptions about $\theta_j$ and type-I vertex conditions, the following holds true.
	\begin{enumerate}
		\item[\textup{(i)}]
		Every pole $\kappa_n$ of $\mu$ is an \emph{anti-resonance}: $T(\kappa_n)=0$. There are
		infinitely many such poles, accumulating at $+\infty$.
		\item[\textup{(ii)}] Set $k_n = \pi n/h$; then for each $n$ with $k_n>\kappa_1$,
		the interval $I_n=(k_n,k_{n+1})$ contains at least $N-1$ \emph{transmission resonances}, i.e.,
		values $k$ with $T(k)=1$.
	\end{enumerate}
	Similar statements hold for type-II vertex conditions.
\end{thm}

\begin{proof}
	\emph{Part~(i).}
    By Corollary~\ref{cor:mu}, $\phi(\kappa_n) ={(2n-1)\pi}/{2}$ and the poles form an unbounded sequence. We have
    $T(\kappa_n)=0$ by Proposition~\ref{pro:zeros-ones}(i) and Remark~\ref{rem:corners}, which establishes part (i).
	
	\emph{Part~(ii).} By Proposition~\ref{pro:zeros-ones}(ii), $T(k)=1$ whenever $U_{N-1}(\xi(k))=0$, i.e., whenever $\xi(k) = A\cos\psi(k)- B\tan\phi(k)\sin\psi(k)$ equals one of the $N-1$ values $\cos(m\pi/N)\in(-1,1)$, $m=1,\dots,N-1$. It suffices to show that $\xi$ takes every value
	in $(-1,1)$ on $I_n$.

    At $k_n = \pi n/h$, we have $\sin\psi(k_n)=0$.
    If $k_n$ is not a pole of $\mu$, then $\xi(k_n) = (-1)^n A$.
    If $k_n$ is a pole of $\mu$ (giving a lattice point $(\phi(k_n),\psi(k_n))\in \Lambda$), then by Lemma~\ref{lem:xi-asymptotics}(ii), 
    \[
    \xi(k_n) = (-1)^n\!\Bigl(A+\frac{Bh}{\phi'(k_n)}\Bigr).
    \]
    In both cases $(-1)^n\xi(k_n) \ge A \ge 1$, so the
    endpoint values alternate in sign and have absolute value at least~$1$.

    To analyse the intermediate values, take $n$ even for the sake of definiteness, so $\xi(k_n)\ge1$ and $\sin\psi>0$ on~$I_n$.
	Let $\kappa$ be the first pole of $\mu$ in
	$I_n$ if one exists, and set $\kappa=k_{n+1}$ otherwise. On $(k_n,\kappa)$, $\phi$ is
	continuous and strictly increasing by Lemma~\ref{lem:phi}, so $\xi$ is continuous
	there. If $\kappa<k_{n+1}$ is a pole of $\mu$, then $\mu(k) = \tan\phi(k)\to+\infty$ as $k\to\kappa^-$, so	$\xi(k)\to-\infty$. If $\kappa=k_{n+1}$, then $\xi(k_{n+1})\le -1$ by the endpoints estimate above. In both cases, the intermediate value
	theorem shows that $\xi$ takes every value in $(-1,1)$ on $(k_n,\kappa)$.
	The odd-$n$ case is treated in the same way.
\end{proof}

%%  ---  Gap-width  --- 

The persistence of anti-resonances, i.e., closed frequencies, as established in Theorem~\ref{thm:high-energy}, raises the question of how strongly $T(k)$ is damped near each such frequency. Figure \ref{Fig5} illustrates two qualitatively different types of zeros of the transmission probability. For a single resonator ($N=1$), all zeros are quadratic, and the graph of $T_1(k)$ exhibits the familiar parabolic profile in their neighbourhoods. For two resonators, however, the zeros at $k = \pi/3$ and $k = \pi$ exhibit markedly different local behaviour. Near $k=\pi$, the transmission probability still vanishes quadratically. However, near $k = \pi/3$, the graph becomes substantially flatter, indicating a higher-order zero. This difference reflects the presence or absence of a collective interference effect generated by the resonator array, as explained in Property~\ref{prp:gap-asympt}. 

\begin{figure}[!h]
	\centering
	\includegraphics[scale=0.28]{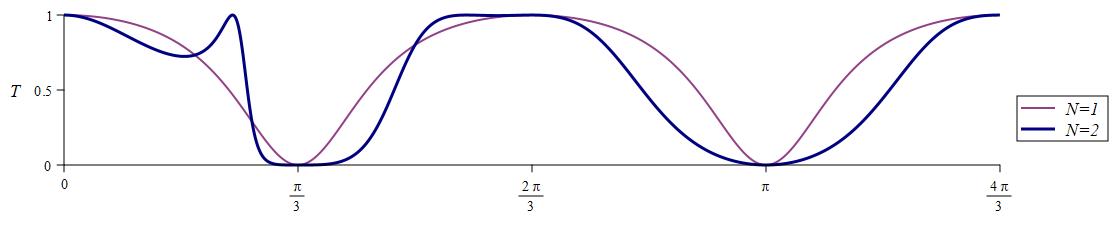}
	\caption{Transmission probability for one and two resonators. All transmission zeros are quadratic for $N=1$. For $N=2$, $k=\pi/3$ is a zero of $T_2$ of generic order $2N=4$, while $k=\pi$ is a zero of $T_2$ of order~$2$ (degenerate case). Parameters: $\theta_1=\theta_2=\theta_3=1$, 
		$\alpha=0$, $\beta=1$, $l=1.5$, and $h=1$.}
	\label{Fig5}
\end{figure}

\begin{prp}[Asymptotics near anti-resonances]\label{prp:gap-asympt}
	At every pole $\kappa>0$ of $\mu$, i.e., at every zero of $T_N$, the following holds.
	\begin{itemize}
		\item[\textup{(i)}] \emph{Generic case:} If $\sin(\kappa h)\ne0$, then
		\begin{equation}\label{eq:gap-generic}
			T_N(k)
			= \frac{[\phi'(\kappa)]^{2N}}
			{4^{N-1}B^{2N}(\sin \kappa h)^{2(N-1)}}
			(k-\kappa)^{2N}
			\bigl(1+o(1)\bigr)
			\quad\text{as }k\to\kappa.
		\end{equation}
		\item[\textup{(ii)}] \emph{Degenerate case:} 
		If $\sin(\kappa h)=0$, then
		\begin{equation}\label{eq:gap-corner}
			T_N(k)
			= \frac{[\phi'(\kappa)]^{2}}
			{B^2 U^2_{N-1}\bigl(A + Bh/\phi'(\kappa)\bigr)}
			(k-\kappa)^{2}
			\bigl(1+o(1)\bigr)
			\quad\text{as }k\to\kappa.
		\end{equation}
	\end{itemize}
	In the case~\textup{(i)}, the vanishing order $2N$
	reflects the collective action of $N$ resonators;
	in the case~\textup{(ii)}, the order drops to $2$
	independently of $N$.
\end{prp}

\begin{proof}
\emph{Case~(i): $\sin(\kappa h)\ne0$.} 
By representation~\eqref{eq:cF-tildeU}, $T_N(k)=\cF(\phi(k),\psi(k))
=\cos^{2N}\!\phi(k)/D(k)$, where
\[
  D(k) = \cos^{2N}\!\phi(k) + \bigl[(A^2-1)\cos^2\!\phi(k) + B^2\sin^2\!\phi(k)\bigr]\,
         \widetilde{U}_{N-1}^2(\phi(k),\psi(k)).
\]
Since $\cos\phi(\kappa) = 0$ and $\phi'(\kappa)>0$ by Lemma~\ref{lem:phi}, we get
\begin{equation}\label{eq:num-expand}
  \cos^{2N}\!\phi(k)
  = [\phi'(\kappa)]^{2N}(k-\kappa)^{2N}\bigl(1+o(1)\bigr).
\end{equation}
By~\eqref{eq:tildeU-trig}, $\widetilde{U}_{N-1}(\phi(\kappa),\psi(\kappa))
= 2^{N-1}(-B\sin\kappa h)^{N-1}\ne0$, whence
$$
    D(k)\to B^2\widetilde{U}_{N-1}^2(\phi(\kappa),
        \psi(\kappa))=4^{N-1}B^{2N}(\sin\kappa h)^{2(N-1)}>0,
$$
and~\eqref{eq:gap-generic} follows.

\smallskip
\emph{Case~(ii): $\sin(\kappa h)=0$.}
Now $(\phi(\kappa),\psi(\kappa))\in\Lambda$, and in view of~\eqref{eq:xi-limit}, $|\xi(k)| \to |\xi(\kappa)| = A + Bh/\phi'(\kappa) > 1$ as $k\to\kappa$. Next, since $U_{N-1}^2$ is an even function,
$U_{N-1}^2(\xi(\kappa))
=U_{N-1}^2\bigl(A+Bh/\phi'(\kappa)\bigr)$. Since $\sin^2\!\phi(\kappa)=1$, the denominator of~$\cF$ in~\eqref{eq:cF} has a non-zero limit 
\[
    B^2 U^2_{N-1}\bigl(A + Bh/\phi'(\kappa)\bigr) 
\]
as $k\to\kappa$, thus resulting in~\eqref{eq:gap-corner}.
\end{proof}

By Property~\ref{prp:gap-asympt}, the local behaviour of $T_N$ near its zeros is determined not
only by the vanishing order but also by the coefficients in~\eqref{eq:gap-generic} and~\eqref{eq:gap-corner}.
In the degenerate case~\eqref{eq:gap-corner}, the coefficient contains $U_{N-1}^2(A+Bh/\phi'(\kappa))$
in the denominator. Since $A+Bh/\phi'(\kappa)>1$, the Chebyshev polynomial $U_{N-1}$ grows rapidly
with $N$ at that point, causing the coefficient of $(k-\kappa)^2$ to decrease and the suppression zone of small $T_N$ to widen substantially as $N$ increases. In the generic case and for large~$N$, the suppression zone is determined by the condition $|k-\kappa|\phi'(\kappa)<2B|\sin(\kappa h)|$; its width may be smaller than for a degenerate quadratic zero. 

This effect is illustrated in Fig.~\ref{Fig6} for $N=8$: the generic zero at $k=\pi/3$ of order~$2N=16$ has asymptotic suppression half-width $[2B\sin(\pi/3)]/\phi'(\kappa)=1/\sqrt{3}\approx 0.58$, while at the degenerate zero at $k=\pi$, the quadratic asymptotics $T_N(k) \approx C (k-\kappa)^2$ with $C =  [\phi'(\pi)]^2/[BU_7(A+Bh/\phi'(\pi))]^2=3^2/U^2_7(4/3)\approx 10^{-4}$ makes the suppression zone much wider.

\begin{figure}[b]
	\centering
	\includegraphics[scale=0.28]{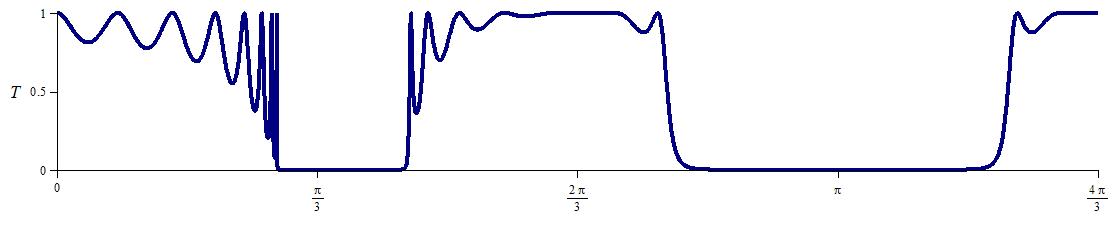}
	\caption{Transmission probability for $N=8$. Although the zero at $k=\pi/3$ has order $16$, the corresponding suppression zone is noticeably narrower than that for the zero $k=\pi$ of $T_8$ of degenerate order~$2$. Parameters: $\theta_1=\theta_2=\theta_3=1$, 
		$\alpha=0$, $\beta=1$, $l=1.5$, and $h=1$.}
	\label{Fig6}
\end{figure}

This raises a natural question: are the suppression
zones near anti-resonances uniformly wide? That is,
given $\varepsilon>0$, does there exist $\delta>0$
such that $T_N(k)<\varepsilon$ throughout a
$\delta$-neighbourhood of every anti-resonance?

For Neumann or Dirichlet conditions with $l/h\in\mathbb{Q}$, the answer is affirmative: by Property~\ref{prp:periodicity}, $T_N$ is periodic, so it suffices to bound below the width of the finitely
many gaps in one period. We show that this fails in the two remaining configurations --- incommensurate lengths $l$ and $h$, or Robin conditions --- using a two-resonator example with $\theta_1=\theta_2=1$, $\theta_3=\sqrt2$,
so that $A=B=1$ and $U_1(\xi)=2\xi$.

\smallskip
\emph{Incommensurate lengths, Neumann conditions.}
Here $\mu(k)=\tan kl$ and
\[
  \xi(k)=\cos kh-\tan(kl)\sin kh
        =\frac{\cos(k(h+l))}{\cos kl},
\]
so by Proposition~\ref{pro:central},
\[
  T_2(k)=\frac{\cos^4 kl}
              {\cos^4 kl+4\sin^2(kl)\cos^2(k(h+l))}.
\]
The anti-resonances are $\kappa_n=(2n-1)\pi/(2l)$
and the transmission resonances are
$k^*_j=(2j-1)\pi/(2(h+l))$, with
\[
  |k^*_j-\kappa_n|
  =\frac{\pi\,|2l(j-n)-h(2n-1)|}{2l(h+l)}.
\]
By Dirichlet's approximation theorem, the
irrationality of $l/(h+l)$ yields integers $j,n$
with $|l(2j+1)-(h+l)(2n-1)|$ arbitrarily small, so
$\inf_{j,n}|k^*_j-\kappa_n|=0$ and no fixed $\delta$
separates anti-resonances from transmission resonances.

\smallskip
\emph{Commensurate lengths, Robin conditions.}
Take $h=2l$ and Robin conditions with
$\alpha,\beta>0$, so that
$\phi(k)=kl-\delta(k)$ with
$\delta(k)=\arctan(\alpha/\beta k)>0$. Then
\[
  \xi(k)=\cos(2kl)-\tan\phi(k)\sin(2kl)
        =\frac{\cos(3kl-\delta(k))}{\cos(kl-\delta(k))}.
\]
The anti-resonances satisfy
$\kappa_n l-\delta(\kappa_n)=(2n-1)\pi/2$, and
the nearest transmission resonance $k^*_n$ satisfies
$3k^*_nl-\delta(k^*_n)=3(2n-1)\pi/2$. Subtracting
three times the first relation from the second,
\[
  k^*_n-\kappa_n
  =\frac{\delta(k^*_n)-3\delta(\kappa_n)}{3l}
  =O(\kappa_n^{-1}),
\]
since $\delta(k)=O(k^{-1})$. Thus the transmission
resonances approach the anti-resonances at rate
$O(1/n)$, and again no fixed $\delta$ suffices.

The Robin and Neumann transmission probabilities differ only through the phase shift $\delta(k)=\arctan(\alpha/\beta k)=O(k^{-1})$, which is negligible at high energy. One might therefore expect $T^{\mathrm{Rob}}$ and $T^{\mathrm{Neu}}$ to be asymptotically close. This is almost true, but fails on a shrinking exceptional set. Near each Neumann anti-resonance $\kappa_n=(2n-1)\pi/(2l)$, the Robin transmission
may sweep from an anti-resonance $T^{\mathrm{Rob}}=0$ to a transmission resonance $T^{\mathrm{Rob}}=1$
over an interval of width $O(1/n)$, as in the example above with $h=2l$. The two functions are
thus far apart on these intervals, however large~$k$ is. Outside such neighbourhoods, however, the
phase shift $\delta(k)$ produces only an $O(\delta(k))=O(1/k)$ change in $\cF$, and the two
transmissions converge uniformly.

\begin{prp}[Robin--Neumann proximity]\label{prp:Robin-Neumann}
Let $\kappa_n=(2n-1)\pi/(2l)$ be the Neumann
anti-resonances. Then for every $\rho>0$,
\[
  \sup_{\substack{k\ge K\\ |k-\kappa_n|\ge\rho\ \forall n}}
  \bigl|T^{\mathrm{Rob}}(k)-T^{\mathrm{Neu}}(k)\bigr|
  \longrightarrow0
  \qquad\text{as }K\to\infty.
\]
\end{prp}

\begin{proof}
By Proposition~\ref{pro:central}, $T^{\mathrm{Neu}}(k)=\cF(\phi^{\mathrm{Neu}}(k),kh)$
and $T^{\mathrm{Rob}}(k)=\cF(\phi^{\mathrm{Rob}}(k),kh)$,
where $\phi^{\mathrm{Neu}}(k)=kl$,
$\phi^{\mathrm{Rob}}(k)=kl-\delta(k)$, and
\[
  \phi^{\mathrm{Rob}}(k)-\phi^{\mathrm{Neu}}(k) =\delta(k)=O(k^{-1}).
\]
Fix $\rho>0$ and suppose $|k-\kappa_n|\ge\rho$ for all $n$. With $\phi^{\mathrm{Neu}}(\kappa_n)
=(2n-1)\pi/2$, the phase $\phi^{\mathrm{Neu}}(k)$ is at distance at least $l\rho$ from $\pi\mathbb{Z}-\tfrac\pi2$; since $\delta(k)\to0$, the same holds for $\phi^{\mathrm{Rob}}(k)$ with $\tfrac12 l\rho$
once $K$ is large enough. Hence both phases belong to the set
\[
  \Omega_\rho=\bigl\{(\phi,\psi):
  \operatorname{dist}(\phi,\pi\mathbb{Z} - \tfrac\pi2)
  \ge\tfrac12 l\rho\bigr\}.
\]
Recall that $\cF$ is $\pi$-periodic in both variables and is continuous on the set $[-\eta,\eta] \times \Real$ for every $\eta \in (0,\pi/2)$; in particular, $\cF$ is uniformly continuous on $\Omega_\rho$. Therefore, for every $\varepsilon >0$ there is $\delta>0$ such that, whenever $\phi_1 - \phi_2 < \delta$ and $(\phi_j, 0) \in \Omega_\rho$, 
\[
    |\cF(\phi_1,\psi) - \cF(\phi_2,\psi)|<\epsilon.
\]
Choosing $K$ large enough so that $|\delta(k)|<\delta$ when $k>K$ results in 
\[
  \bigl|T^{\mathrm{Rob}}(k)-T^{\mathrm{Neu}}(k)\bigr| 
  =  \bigl|\cF(\phi^{\mathrm{Rob}}(k),\psi(k))-\cF(\phi^{\mathrm{Neu}}(k),\psi(k))\bigr|    
  <\varepsilon
\]
whenever $|k-\kappa_n|\ge\rho$ for every $n\in\mathbb{N}$. The proof is complete.
\end{proof}

\begin{rem}\rm
The exclusion of the neighbourhoods $\{|k-\kappa_n|<\rho\}$ is essential: the convergence does not hold on all of $[K,\infty)$. For $h=2l$ the above construction produces transmission resonances $k^*_n\to\infty$ with
$T_2^{\mathrm{Rob}}(k^*_n)=1$, while $k^*_n$ lies within $O(1/n)$ of the Neumann anti-resonance
$\kappa_n$, where $T_2^{\mathrm{Neu}}$ vanishes
quadratically, so $T_2^{\mathrm{Neu}}(k^*_n)=O(1/n^2)$.
Hence $|T_2^{\mathrm{Rob}}(k^*_n)-T_2^{\mathrm{Neu}}(k^*_n)|
\to1$ and 
\[
  \limsup_{k\to\infty}
  \bigl|T_2^{\mathrm{Rob}}(k)-T_2^{\mathrm{Neu}}(k)\bigr|=1.
\]
Therefore, the Robin and Neumann transmission probabilities approach each other as $k\to\infty$ uniformly away from the anti-resonances, but the difference does not tend to zero on a full neighbourhood of infinity.
\end{rem}

%% ==========================================
\section{Large-$N$ limit and analysis of parameter dependence}
\label{sec:asymptotics}
%% ==========================================

In this section, we illustrate the sensitivity of $T_N(k)$ to the model parameters through several representative studies, isolating the effects of (i)~the number of resonators~$N$, (ii)~the coupling strength~$\theta_3$, (iii)~the dipole interaction ratio~$\theta_2/\theta_1$, and (iv) the resonator length and overall filter design. Throughout, we take type-I vertex conditions and Neumann boundary conditions at the vertices~$u_j$.

%% ==========================================
\subsection{Infinite-array limit and band structure}
\label{ssec:Ntoinfty}
%% ==========================================

We first recall the band--gap structure of the fully periodic quantum graph $H_\infty$, and then show how it governs the regions of high and low transmission of~$T_N$ as $N\to\infty$.

\smallskip
\noindent\textit{The periodic array.}
Let $G_\infty$ be the infinite periodic graph obtained by attaching identical resonators of length~$l$ at every
integer multiple of~$h$ on the real line, with the same vertex conditions as in Section~\ref{sec:local-periodic}. By the Bloch--Floquet theorem~\cite{BerkolaikoKuchment2013}, the spectrum of the associated Hamiltonian $H_\infty$ is determined by the single-cell transfer matrix $Q(k)$, whose characteristic
equation is $\lambda^2-2\xi(k)\lambda+1=0$ with roots $\lambda_\pm=\xi\pm\sqrt{\xi^2-1}$ and half-trace
$\xi(k)=A\cos kh-B\mu(k)\sin kh$. The spectrum of~$H_\infty$ is governed entirely by $\xi$:
\begin{itemize}
  \item if $|\xi(k)|\le1$, then $\lambda_\pm=e^{\pm iKh}$ for a real Bloch quasimomentum $K\in[0,\pi/h]$, and
  $k^2$ lies in a \emph{pass-band}; the union of the pass-bands is the (absolutely continuous) spectrum
  of~$H_\infty$;
  \item if $|\xi(k)|>1$, then $\lambda_\pm$ are real with $|\lambda_+|>1>|\lambda_-|$, the Bloch solutions grow or decay exponentially, and $k^2$ lies in a \emph{band gap}, contained in the resolvent set of~$H_\infty$.
\end{itemize}
In the pass-bands the dispersion relation reads 
\begin{equation}\label{eq:dispersion}
  \cos Kh = \xi(k)
  = \frac{(\theta_1^2+\theta_2^2)\cos kh
    - \theta_3^2\,\mu(k)\sin kh}{2\theta_1\theta_2},
\end{equation}
the quantum-graph analogue of the Kronig--Penney relation~\cite{GriffithsSteinke}; the resonators enter through
$\mu(k)$, producing energy-dependent gaps absent in the classical model.

\smallskip
\noindent\textit{Finite arrays and the formation of bands.}
The finite-$N$ transmission probability fits into this picture through the same half-trace $\xi$, and the
band/gap dichotomy explains the sharpening of $T_N$ observed numerically, see Fig.~\ref{Fig8} for $N=1,2,6$.

\begin{figure}[h]
	\centering
	\includegraphics[scale=0.32]{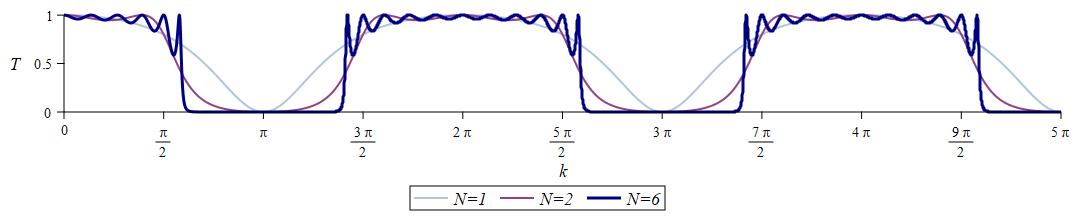}
	\caption{Plots of $T_N(k)$ for $N=1,2,6$. Parameters: $\theta_1=\theta_2=\theta_3=1$, 
		 $l=0.5$, and $h=1$.}
	\label{Fig8}
\end{figure}

\emph{Reflection resonances lie in the gaps.} By Theorem~\ref{thm:high-energy}, the zeros of $T_N$ occur at
the poles $\kappa_n$ of $\mu$, independently of~$N$. If $\sin(\kappa_n h)\ne0$ then $\kappa_n$ is a pole of~$\xi$, so $|\xi(\kappa_n)|=\infty$; if $\sin(\kappa_n h)=0$ then Lemma~\ref{lem:xi-asymptotics}(ii) gives
$|\xi(\kappa_n)|>1$. In either case $|\xi(\kappa_n)|>1$, so every reflection resonance lies strictly inside a band gap of~$H_\infty$, for every~$N$.

\emph{Transmission resonances fill the pass-bands.} By
Proposition~\ref{pro:zeros-ones}, $T_N(k)=1$ when $\xi(k)=\cos(\pi m/N)$ for some $m=1,\dots,N-1$ (and, for $\theta_1=\theta_2$, also when $\sin\phi(k)=0$, where $\xi(k)=\cos kh$). These values lie in $[-1,1]$, so every
transmission resonance lies in a pass-band. As $N\to\infty$ the values $\cos(\pi m/N)$ become dense in~$[-1,1]$, and the transmission resonances densely fill the pass-bands.

\emph{Behaviour at the band edges.} At a band edge~$k_0$, where $|\xi(k_0)|=1$, the Chebyshev polynomial
satisfies $|U_{N-1}(\xi(k_0))|=|U_{N-1}(\pm1)|=N$. If $z(k_0)\ne0$, then 
\[
  T_N(k_0)=\frac{1}{1+|z(k_0)|^2N^2}
  =\frac{1}{|z(k_0)|^2}\,N^{-2}\bigl(1+o(1)\bigr),
\]
a quadratic decay in~$N$; if $z(k_0)=0$, then $T_N(k_0)=1$ for all~$N$, and the band edge supports perfect transmission.

\smallskip
Together these mechanisms give a complete picture of the large-$N$ behaviour: inside the gaps $T_N$ decays
exponentially in~$N$, at the band edges it decays like $N^{-2}$ (unless $z=0$ there), and inside the pass-bands the perfect-transmission resonances become dense. For large arrays this produces the filter-like profile of~Fig.~\ref{Fig1} ($N=20$): wide intervals of near-total reflection, coinciding with the gaps of~$H_\infty$, alternating with pass-bands in which $T_N\approx1$. Theorem~\ref{thm:high-energy} and Property~\ref{prp:gap-asympt} of the previous section thus acquire a spectral interpretation in the $N\to\infty$ limit.

%% ==========================================
\subsection{Dependence on vertex parameters}
%
%% ==========================================
Finally, we briefly discuss how the interaction parameters $\theta_j$ affect the system's scattering profile. 

\smallskip 
\noindent\textit{Effect of coupling strength $\theta_3$.} 
We set $\theta_1=\theta_2=1$ and vary $\theta_3$. When $\theta_3=0$, the interaction between the scattering channel and the resonators is suppressed, and $T_N(k)=1$ identically, as immediately follows from~\eqref{eq:TN-A}. Figure~\ref{Fig9} shows $T_N(k)$ for $\theta_3=0.5,\,1.5$, and $3$ with $N=6$. As~$\theta_3$ increases from zero, reflection windows emerge 
near the poles of $\mu(k)$ at $k=\pi(2n-1)/(2l)$. These dips deepen and 
widen with $\theta_3$, while the transmission between them acquires 
a fine oscillatory structure reflecting the increasing strength of 
multiple scattering between the channel and the resonator array.

\begin{figure}[!h]
	\centering
	\includegraphics[scale=0.33]{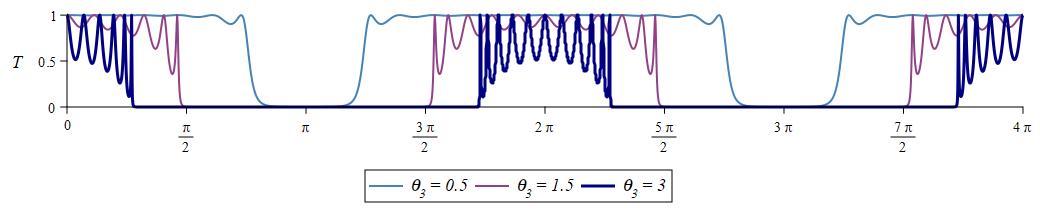}
	\caption{Plots of $T_N(k)$ for increasing values 
		of the coupling parameter $\theta_3$. Parameters: $\theta_1=\theta_2=1$, 
		 $l=0.5$, $h=1$, and $N=6$.}
	\label{Fig9}
\end{figure}

\smallskip
\noindent\textit{Effect of dipole interaction ratio $\theta_2/\theta_1$.}
We set $\theta=(1,\theta_2,1)$ and vary $\theta_2$. Figure~\ref{Fig10} shows $T_N(k)$ for $\theta_2=1.5,\,2,$ and $6$. Unlike in the previous scenarios, $T_N$ is no longer uniformly close to~$1$. As $\theta_2/\theta_1$ grows, the profile rapidly develops a structure resembling a $\delta'$-comb~\cite{GolovatyHrynivLavrynenko2025}: narrow transmission peaks separated by broad reflection regions. 

\begin{figure}[!h]
	\centering
	\includegraphics[scale=0.34]{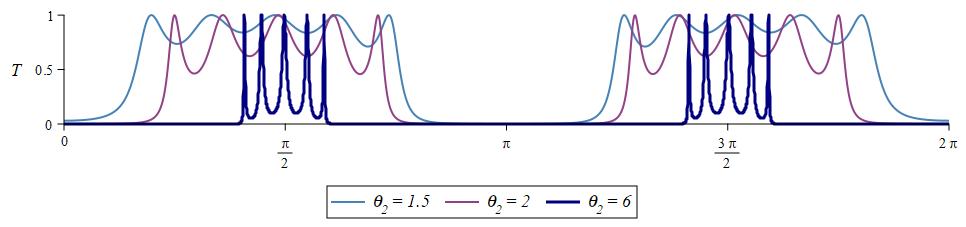}
	\caption{Plots of $T_N(k)$ for increasing ratio 
		$\theta_2/\theta_1$. Parameters: $\theta_1=\theta_3=1$, 
		 $l=0.5$, $h=1$, and $N=6$.}
	\label{Fig10}
\end{figure}

Formula~\eqref{eq:TN-A} sheds light on the formation of $\delta'$-comb structure. When $\theta_3^2\ll\theta_1^2+\theta_2^2$, the $\mu$-dependent terms in both $|z(k)|^2$ and $\xi(k)$ become negligible away from the poles of~$\mu$, and $\xi(k)$ is well approximated there by $\frac{\theta_1^2+\theta_2^2}{2\theta_1\theta_2}\cos kh$, 
which is precisely the half-trace of the $\delta'$-comb transfer matrix~\cite{GolovatyHrynivLavrynenko2025}. In this regime, $T_N$ is well approximated away from the anti-resonances by its $\theta_3=0$ limit, the transmission probability of the $\delta'$-comb. 

\smallskip
\noindent\textit{Effect of changing the resonator length~$l$.}
The scattering profile is sensitive to the full vector $\theta$ and to the resonator length, and tuning these parameters produces a wide variety of spectral filters. Figures~\ref{Fig11} and~\ref{Fig12} illustrate this for $\theta=(1,4,6)$ at two resonator lengths. Increasing~$l$ from $0.5$ to $2$ at fixed spacing $h=1$ moves the poles of~$\mu$ (the anti-resonances) at $\kappa_n=(2n-1)\pi/(2l)$ closer together, so the longer resonators produce more reflection windows per unit wavenumber and a correspondingly finer band--gap pattern.

\begin{figure}[h]
	\centering
	\includegraphics[scale=0.34]{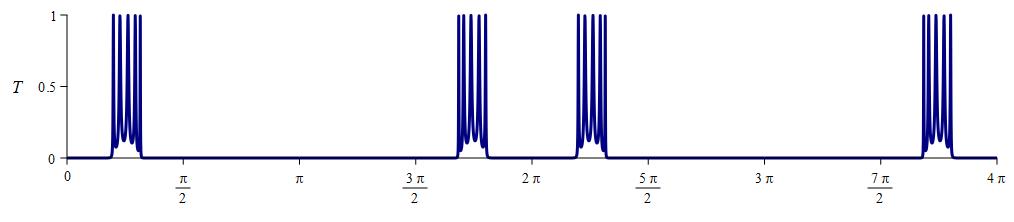}
	\caption{Transmission probability for $\theta_1=1$, $\theta_2=4$, $\theta_3=6$, $l=0.5$, $h=1$, and $N=6$. Compare with $l=2$ case in Fig.~\ref{Fig12}.}
	\label{Fig11}
\end{figure}
\begin{figure}[h]
	\centering
	\includegraphics[scale=0.56]{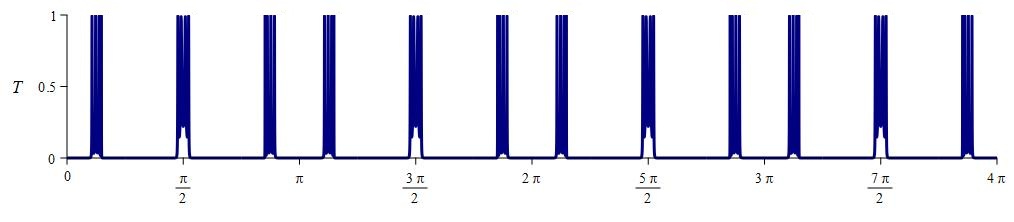}
	\caption{Transmission probability for $\theta_1=1$, $\theta_2=4$, $\theta_3=6$,  $l=2$, $h=1$, and $N=6$. Compare with $l=0.5$ case in Fig.~\ref{Fig11}.}
	\label{Fig12}
\end{figure}

%% ==========================================
\section{Conclusions and discussion}
\label{sec:conclusions}
%% ==========================================

We have studied quantum scattering on a line carrying a finite locally periodic array of resonators attached through scale-invariant vertex couplings. Reducing the problem to one with energy-dependent point interactions, we obtained a closed-form expression for the transmission probability $T_N$ (Theorem~\ref{thm:TN}) and recast it as $T_N=\cF(\phi,\psi)$, a single function of the resonator phase~$\phi$ and the propagation phase~$\psi$ (Proposition~\ref{pro:central}). This representation is the basis for the entire analysis: it yielded the symmetries and invariances of~$T_N$, its (quasi-)periodicity (Property~\ref{prp:periodicity}), and the location of its zeros and ones (Proposition~\ref{pro:zeros-ones}).

The central message is that scale invariance removes the high-energy transparency of ordinary one-dimensional
scattering. The transmission probability retains infinitely many anti-resonances and transmission resonances at arbitrarily high energy (Theorem~\ref{thm:high-energy}); near each anti-resonance it is suppressed at a rate $(k-\kappa)^{2N}$ set by the array, dropping to $(k-\kappa)^2$ at the exceptional energies where the collective effect is absent (Property~\ref{prp:gap-asympt}). For Robin endpoint conditions the transmission approaches the Neumann one away from the anti-resonances, but not uniformly near them (Property~\ref{prp:Robin-Neumann}). In the large-$N$ limit these features acquire a spectral interpretation: the anti-resonances lie in the gaps and
the transmission resonances fill the bands of the limiting periodic operator~$H_\infty$, so that the array realises a
tunable spectral filter.

The model extends in several directions. For instance, the channel may carry, instead of a single resonator at
each junction, a fixed cluster of resonators of different lengths and endpoint conditions, repeated periodically. Since resonators meeting at a common vertex contribute additively to the Dirichlet-to-Neumann data, the analysis extends with only notational changes once $\mu$ is replaced by the weighted sum
\[
  \mu_{\mathrm{eff}}(k)
  = \sum_{j=3}^{K}\theta_j^2\,
    \mu(k,l_j,\alpha_j,\beta_j),
\]
with $\mu(k,l,\alpha,\beta)$ given by~\eqref{eq:mu}. All results of Sections~\ref{sec:properties}
and~\ref{sec:asymptotics} hold with $\mu_{\mathrm{eff}}$ in place of~$\mu$, while the additional parameters of the cluster provide further freedom for fine-tuning the transmission profile.

A natural further direction is to treat genuinely aperiodic or disordered arrays, in which the resonator
lengths or couplings vary from cell to cell. The transfer-matrix product then no longer reduces to a single
matrix power, and the interplay between the scale invariance of the couplings and the disorder---in
particular, whether the high-energy persistence of reflection survives---remains to be understood.

\end{document}